\documentclass{amsart}
\usepackage{amsfonts}

\newtheorem{theorem}{Theorem}[section]

\newtheorem{corollary}[theorem]{Corollary}
\newtheorem{lemma}[theorem]{Lemma}
\newtheorem{example}[theorem]{Example}
\newtheorem{definition}[theorem]{Definition}
\newtheorem{proposition}[theorem]{Proposition}
\newtheorem{remark}[theorem]{Remark}

\numberwithin{theorem}{section}
\input{tcilatex}

\begin{document}
\title[Group actions and their pro-$C^{\ast }$-crossed products]{Group
actions on pro-$C^{\ast }$-algebras and their pro-$C^{\ast }$-crossed
products}
\author{Maria Joi\c{t}a}
\address{Department of Mathematics, Faculty of Mathematics and Computer
Science, University of Bucharest, Str. Academiei nr.14, Bucharest, Romania}
\email{ mjoita@fmi.unibuc.ro}
\address{Simion Stoilow Institute of Mathematics of the Roumanian Academy,
21 Calea Grivitei Street, 010702 Bucharest, Romania}
\urladdr{http://sites.google.com/a/g.unibuc.ro/maria-joita/}
\thanks{Supported by a grant of the Romanian National Authority for
Scientific Research, CNCS - UEFISCDI, project number PN-II-ID-PCE-2012-4-0201%
}

\begin{abstract}
In this paper, we define the notions of full pro-$C^{\ast }$-crossed
product, respectively reduced pro-$C^{\ast }$-crossed product, of a pro-$%
C^{\ast }$-algebra $A\left[ \tau _{\Gamma }\right] $ by a strong bounded
action $\alpha $ of a locally compact group $G$ and investigate some of
their properties.
\end{abstract}

\maketitle

\section{Introduction}

Given a $C^{\ast }$-algebra $A$ and a continuous action $\alpha $ of a
locally compact group $G$ on $A$, we can construct a new $C^{\ast }$%
-algebra, called the crossed product of $A$ by $\alpha $, usually denoted by 
$G\times _{\alpha }A$, and which contains, in some subtle sense, $A$ and $G$%
. The origin of this construction goes back to Murray and von Neumann and
their group measure space construction by which they associated a von
Neumann algebra to a countable group acting on a measure space. The analog
of this construction for the case of $C^{\ast }$-algebras is due to Gelfand
with co-authors Naimark and Fomin. There is a vast literature on crossed
products of $C^{\ast }$-algebras (see, for example, \cite{W}), but the
corresponding theory in the context of non-normed topological $\ast $%
-algebras has still a long way to go.

Crossed product of pro-$C^{\ast }$-algebras by inverse limit actions of
locally compact groups were considered by Phillips \cite{P2} and Joi\c{t}a 
\cite{J2, J3, J4}. If $\alpha $ is an inverse limit action of a locally
compact group $G$ on a pro-$C^{\ast }$-algebra $A\left[ \tau _{\Gamma }%
\right] $ whose topology is given by the family of $C^{\ast }$-seminorms $%
\Gamma =\{p_{\lambda }\}_{\lambda \in \Lambda }$, then the covariance
algebra $L^{1}(G,\alpha ,A\left[ \tau _{\Gamma }\right] )$ has a structure
of locally $m$-convex $\ast $-algebra with topology given by the family of
submultiplicative seminorms $\{N_{p_{\lambda }}\}_{\lambda \in \Lambda }$,
where%
\begin{equation*}
N_{p_{\lambda }}\left( f\right) =\tint\limits_{G}p_{\lambda }\left( f\left(
g\right) \right) dg
\end{equation*}%
and the full crossed product of $A\left[ \tau _{\Gamma }\right] $ by $\alpha 
$ is defined as the enveloping pro-$C^{\ast }$-algebra of $L^{1}(G,\alpha ,A%
\left[ \tau _{\Gamma }\right] )$. In particular, for a given inverse limit
automorphism $\alpha $ of a pro-$C^{\ast }$-algebra $A\left[ \tau _{\Gamma }%
\right] $, we can associate to the pair $(A\left[ \tau _{\Gamma }\right]
,\alpha )$ a pro-$C^{\ast }$-algebra by the crossed product construction,
but if $\alpha $ is not an inverse limit automorphism, this construction is
not possible because the covariance algebra has not a structure of locally $%
m $-convex $\ast $-algebra ($N_{p_{\lambda }}$ is not a submultiplicative $%
\ast $-seminorm). On the other hand, a transformation group $(X,G)$, with $X$
a countably compactly generated Hausdorff topological space ($X$ is a direct
limit of a countable family of compact spaces $\{K_{n}\}_{n}$), induces an
action $\alpha $ of $G$ on the pro-$C^{\ast }$-algebra $C(X)$, which is not
in general an inverse limit action and so, we can not associate to $(X,G)$ a
pro-$C^{\ast }$-algebra by the above construction.

It is well known that the crossed product of $C^{\ast }$-algebras is a
universal object for nondegenerate covariant representations (see, for
example, \cite{R}). In \cite{J3}, we show that the crossed product of pro-$%
C^{\ast }$-algebras by inverse limit actions has also the universal property
with respect to the nondegenerate covariant representations. In this paper,
we define the full crossed product of a pro-$C^{\ast }$-algebra $A\left[
\tau _{\Gamma }\right] $ by an action $\alpha $ of a locally compact group $%
G $ as a universal object for nondegenerate covariant representations and we
show that the full crossed product of pro-$C^{\ast }$-algebras exists for
strong bounded actions. Strong boundless of the action $\alpha $ is
essential to prove the existence of a covariant representation.
Unfortunately, if the action $\alpha $ of $G$ on $A\left[ \tau _{\Gamma }%
\right] $ is strongly bounded, then there is another family of $C^{\ast }$%
-seminorms on $A\left[ \tau _{\Gamma }\right] \ $which induces the same
topology on $A$, and $\alpha $ is an inverse limit action with respect to
this family of $C^{\ast }$-seminorms.

The organization of this paper is as follows. After preliminaries in Section
2, we introduce the notion of strong bounded action of a locally compact
group $G$ on a pro-$C^{\ast }$-algebra $A\left[ \tau _{\Gamma }\right] $ and
present some examples in Section 3. We show that there is a nondegenerate
covariant representation of a pro-$C^{\ast }$-algebra $A\left[ \tau _{\Gamma
}\right] $ with respect to a strong bounded action $\alpha $ of a locally
compact group $G$ on $A\left[ \tau _{\Gamma }\right] $ in Section 4. In
Section 5, we show that if $\alpha $ is strongly bounded, then there exists
the full pro-$C^{\ast }$-crossed product of $A\left[ \tau _{\Gamma }\right] $
by $\alpha $. Also, we show that the full pro-$C^{\ast }$-crossed product is
invariant under the conjugacy of the actions, and the full pro-$C^{\ast }$%
-crossed product of the maximal tensor product $A\left[ \tau _{\Gamma }%
\right] \otimes _{\max }B\left[ \tau _{\Gamma ^{\prime }}\right] $ by the
action $\alpha \otimes $id is isomorphic to the maximal tensor product of
the full pro-$C^{\ast }$-crossed product of $A\left[ \tau _{\Gamma }\right] $
by $\alpha $ and $B\left[ \tau _{\Gamma ^{\prime }}\right] $. In Section 6,
we define the notion of reduced pro-$C^{\ast }$-crossed product of a pro-$%
C^{\ast }$-algebra $A\left[ \tau _{\Gamma }\right] $ by a strong bounded
action $\alpha $, and show that the reduced pro-$C^{\ast }$-crossed product
is invariant under the conjugacy of the actions. Also, we show that the
reduced pro-$C^{\ast }$-crossed product of the minimal tensor product $A%
\left[ \tau _{\Gamma }\right] \otimes _{\min }B\left[ \tau _{\Gamma ^{\prime
}}\right] $ by the action $\alpha \otimes $id is isomorphic to the minimal
tensor product of the reduced pro-$C^{\ast }$-crossed product of $A\left[
\tau _{\Gamma }\right] $ by $\alpha $ and $B\left[ \tau _{\Gamma ^{\prime }}%
\right] $. It is known that the full crossed product of a $C^{\ast }$%
-algebra by an action $\alpha $ of an amenable locally compact group $G$
coincides with the reduced crossed product. We show that this result is
still valid for pro-$C^{\ast }$-crossed products.

\section{Preliminaries}

A seminorm $p$\ on a topological $\ast $-algebra $A$\ satisfies the $C^{\ast
}$-condition (or is a $C^{\ast }$-seminorm) if $p\left( a^{\ast }a\right)
=p\left( a\right) ^{2}$ for all $a\in A$. It is known that such a seminorm
must be submultiplicative ($p\left( ab\right) \leq p\left( a\right) p\left(
b\right) $ for all $a,b\in A$) and $\ast $-preserving ($p\left( a^{\ast
}\right) =p\left( a\right) $ for all $a\in A$).

A pro-$C^{\ast }$-algebra is a complete Hausdorff topological $\ast $%
-algebra $A$\ whose topology is given by a directed family of $C^{\ast }$%
-seminorms $\{p_{\lambda };\lambda \in \Lambda \}$. Other terms used for pro-%
$C^{\ast }$-algebras are: locally $C^{\ast }$-algebras (A. Inoue, M.
Fragoulopoulou, A. Mallios, etc.), $LMC^{\ast }$-algebras (G. Lassner, K.
Schm\"{u}dgen), $b^{\ast }$-algebras (C. Apostol).

Let $A\left[ \tau _{\Gamma }\right] $\ be a pro-$C^{\ast }$-algebra with
topology given by $\Gamma =\{p_{\lambda };\lambda \in \Lambda \}$ and let $B%
\left[ \tau _{\Gamma ^{\prime }}\right] $\ be a pro-$C^{\ast }$-algebra with
topology given by $\Gamma ^{\prime }=\{q_{\delta };\delta \in \Delta \}$. A
continuous $\ast $-morphism $\varphi :A\left[ \tau _{\Gamma }\right]
\rightarrow B\left[ \tau _{\Gamma ^{\prime }}\right] $ (that is, $\varphi $
is linear, $\varphi \left( ab\right) =\varphi (a)\varphi (b)$ and $\varphi
(a^{\ast })=\varphi (a)^{\ast }$ for all $a,b\in A$ and for each $q_{\delta
}\in \Gamma ^{\prime },$ there is $p_{\lambda }\in \Gamma $ such that $%
q_{\delta }\left( \varphi (a)\right) \leq p_{\lambda }\left( a\right) $ for
all $a\in A$) is called a pro-$C^{\ast }$-morphism. The pro-$C^{\ast }$%
-algebras $A\left[ \tau _{\Gamma }\right] $ and $B\left[ \tau _{\Gamma
^{\prime }}\right] $ are isomorphic if there is a pro-$C^{\ast }$%
-isomorphism $\varphi :A\left[ \tau _{\Gamma }\right] \rightarrow B\left[
\tau _{\Gamma ^{\prime }}\right] $ (that is, $\varphi $ is invertible, $%
\varphi $ and $\varphi ^{-1}$ are pro-$C^{\ast }$-morphisms).

If $\{A_{\lambda };\pi _{\lambda \mu }\}_{\lambda \geq \mu ,\lambda ,\mu \in
\Lambda }$ is an inverse system of $C^{\ast }$-algebras, then $%
\lim\limits_{\leftarrow \lambda }A_{\lambda }$ with topology given by the
family of $C^{\ast }$-seminorms $\{p_{\lambda }\}_{\lambda \in \Lambda },$
with $p_{\lambda }\left( \left( a_{\mu }\right) _{\mu \in \Lambda }\right)
=\left\Vert a_{\lambda }\right\Vert _{A_{\lambda }}$ for all $\lambda \in
\Lambda $, is a pro-$C^{\ast }$-algebra.

Let $A\left[ \tau _{\Gamma }\right] $\ be a pro-$C^{\ast }$-algebra with
topology given by $\Gamma =\{p_{\lambda };\lambda \in \Lambda \}$. For $%
\lambda \in \Lambda $,\ $\ker p_{\lambda }$\ is a closed $\ast $-bilateral
ideal and $A_{\lambda }=A/\ker p_{\lambda }$\ is a $C^{\ast }$-algebra in
the $C^{\ast }$-norm $\left\Vert \cdot \right\Vert _{p_{\lambda }}$\ induced
by $p_{\lambda }$\ (that is, $\left\Vert a\right\Vert _{p_{\lambda }}=$\ $p_{%
{\small \lambda }}(a),$ for all $a\in A$). The canonical map from $A$ to $%
A_{\lambda }$ is denoted by $\pi _{\lambda }^{A},$ $\pi _{\lambda
}^{A}\left( a\right) =a+\ker p_{\lambda }$ for all $a\in A$. For $\lambda
,\mu \in \Lambda $\ with $\mu \leq \lambda $\ there is a surjective $C^{\ast
}$-morphism $\pi _{\lambda \mu }^{A}:A_{\lambda }\rightarrow A_{\mu }$\ such
that $\pi _{\lambda \mu }^{A}\left( a+\ker {\small p}_{\lambda }\right)
=a+\ker p_{\mu }$, and then $\{A_{\lambda };\pi _{\lambda \mu
}^{A}\}_{\lambda ,\mu \in \Lambda }$\ is an inverse system of $C^{\ast }$%
-algebras. Moreover, pro-$C^{\ast }$-algebras$\ A\left[ \tau _{\Gamma }%
\right] $ and $\lim\limits_{\leftarrow \lambda }A_{\lambda }$ are isomorphic 
$\ $(Arens-Michael decomposition).

Let $\{\left( \mathcal{H}_{\lambda },\left\langle \cdot ,\cdot \right\rangle
_{\lambda }\right) \}_{\lambda \in \Lambda }$\ be a family of Hilbert spaces
such that $\mathcal{H}_{\mu }\subseteq \mathcal{H}_{\lambda }$ and $%
\left\langle \cdot ,\cdot \right\rangle _{\lambda }|_{\mathcal{H}_{\mu }}$ $%
= $\ $\left\langle \cdot ,\cdot \right\rangle _{\mu }$ for all $\lambda ,\mu
\in \Lambda $ with $\mu \leq \lambda $. $\mathcal{H}=\lim\limits_{\lambda
\rightarrow }\mathcal{H}_{\lambda }$\ with inductive limit topology is
called a locally Hilbert space.

Let $L(\mathcal{H})=\{T:\mathcal{H}\rightarrow \mathcal{H};T_{\lambda }=T|_{%
\mathcal{H}_{\lambda }}\in L(\mathcal{H}_{\lambda })$\ and $P_{\lambda \mu
}T_{\lambda }=T_{\lambda }P_{\lambda \mu }$ for\ \ all $\lambda ,\mu \in
\Lambda $\ with $\mu \leq \lambda \}$, where $P_{\lambda \mu }$ is the
projection of $\mathcal{H}_{{\small \lambda }}$ on $\mathcal{H}_{{\small \mu 
}}$. Clearly, $L(\mathcal{H})$\ is an algebra in an obvious way, and $%
T\rightarrow T^{\ast }$\ with $T^{\ast }|_{\mathcal{H}_{\lambda }}=\left(
T_{\lambda }\right) ^{\ast }$\ for all $\lambda \in \Lambda $\ is an
involution.

For each $\lambda \in \Lambda $,\ the map $p_{\lambda ,L(\mathcal{H})}:L(%
\mathcal{H})\rightarrow \lbrack 0,\infty )$\ given by $p_{\lambda ,L(%
\mathcal{H})}\left( T\right) =\left\Vert T|_{\mathcal{H}_{\lambda
}}\right\Vert _{L(\mathcal{H}_{\lambda })}$\ is a $C^{\ast }$-seminorm on $L(%
\mathcal{H})$, and with topology given by the family of $C^{\ast }$%
-seminorms $\{p_{\lambda ,L(\mathcal{H})}\}_{\lambda \in \Lambda }$,\ $L(%
\mathcal{H})$ becomes a pro-$C^{\ast }$-algebra.

$L(\mathcal{H})$ as a pro-$C^{\ast }$-algebra has an Arens-Michael
decomposition, given by the $C^{\ast }$-algebras $L(\mathcal{H})_{\lambda
}=L(\mathcal{H})/\ker p_{\lambda }$, $\lambda \in \Lambda $. Moreover, for
each $\lambda \in \Lambda $, the map $\varphi _{{\small \lambda }}:L(%
\mathcal{H})_{{\small \lambda }}\rightarrow L(\mathcal{H}_{{\small \lambda }%
})\ $given by $\varphi _{{\small \lambda }}\left( T+\ker p_{\lambda }\right)
=T|_{\mathcal{H}_{\lambda }}$ is an isometric $\ast $-morphism. The
canonical maps from $L(\mathcal{H})$ to $L(\mathcal{H})_{\lambda },\lambda
\in \Lambda $ are denoted by $\pi _{\lambda }^{\mathcal{H}},\lambda \in
\Lambda $, and $\pi _{{\small \lambda }}^{\mathcal{H}}(T)=T|_{\mathcal{H}%
_{\lambda }}$. \ For a given pro-$C^{\ast }$-algebra $A\left[ {\small \tau }%
_{\Gamma }\right] $\ there is a locally Hilbert space $\mathcal{H}$ such
that $A\left[ {\small \tau }_{\Gamma }\right] $ is isomorphic to a pro-$%
C^{\ast }$-subalgebra of $L(\mathcal{H})$ (see \cite[Theorem 5.1]{I}).

A multiplier of $A\left[ {\small \tau }_{\Gamma }\right] $ is a pair $\left(
l,r\right) $ of linear maps $l,r:A\left[ {\small \tau }_{\Gamma }\right]
\rightarrow A\left[ {\small \tau }_{\Gamma }\right] $ such that are
respectively left and right $A$-module homomorphisms and $r(a)b=al(b)$ for
all $a,b\in A$. The set $M(A\left[ {\small \tau }_{\Gamma }\right] )$ of all
multipliers of $A\left[ {\small \tau }_{\Gamma }\right] $ is a pro-$C^{\ast
} $-algebra with multiplication given by $\left( l_{1},r_{1}\right) \left(
l_{2},r_{2}\right) =\left( l_{1}l_{2},r_{2}r_{1}\right) $,$\ $ the
involution given by $\left( l,r\right) ^{\ast }=\left( r^{\ast },l^{\ast
}\right) $, where $r^{\ast }\left( a\right) =r\left( a^{\ast }\right) ^{\ast
}$ and $l^{\ast }\left( a\right) =l\left( a^{\ast }\right) ^{\ast }$ for all 
$a\in A$, and the topology given by the family of $C^{\ast }$-seminorms $%
\{p_{\lambda ,M(A\left[ {\small \tau }_{\Gamma }\right] )}\}_{\lambda \in
\Lambda }$, where $p_{\lambda ,M(A\left[ {\small \tau }_{\Gamma }\right]
)}\left( l,r\right) =\sup \{p_{\lambda }(l(a));p_{\lambda }(a)\leq 1\}$.
Moreover, for each $p_{\lambda }\in \Gamma $, the $C^{\ast }$-algebras $%
\left( M(A\left[ {\small \tau }_{\Gamma }\right] )\right) _{\lambda }$ and $%
M(A_{\lambda })$ are isomorphic. The strict topology on $M(A\left[ {\small %
\tau }_{\Gamma }\right] )$ is given by the family of seminorms $\{p_{\lambda
,a}\}_{\left( \lambda ,a\right) \in \Lambda \times A}$, where $p_{\lambda
,a}\left( l,r\right) =p_{\lambda }\left( l\left( a\right) \right)
+p_{\lambda }\left( r\left( a\right) \right) $, $M(A\left[ {\small \tau }%
_{\Gamma }\right] )$ is complete with respect to the strict topology and $A%
\left[ {\small \tau }_{\Gamma }\right] $ is dense in $M(A\left[ {\small \tau 
}_{\Gamma }\right] )$ (see \cite{P1} and \cite[Proposition 3.4]{J1}).

A pro-$C^{\ast }$-morphism $\varphi :A\left[ {\small \tau }_{\Gamma }\right]
\rightarrow M(B\left[ \tau _{\Gamma ^{\prime }}\right] )$ is nondegenerate
if $\left[ \varphi \left( A\right) B\right] =B\left[ \tau _{\Gamma ^{\prime
}}\right] $, where $\left[ \varphi \left( A\right) B\right] $ denotes the
closed subspace of $B\left[ \tau _{\Gamma ^{\prime }}\right] $ generated by $%
\{\varphi \left( a\right) b;$ $a\in A,b\in B\}$. A nondegenerate pro-$%
C^{\ast }$-morphism $\varphi :A\left[ {\small \tau }_{\Gamma }\right]
\rightarrow M(B\left[ \tau _{\Gamma ^{\prime }}\right] )$ extends to a
unique pro-$C^{\ast }$-morphism $\overline{\varphi }:M(A\left[ {\small \tau }%
_{\Gamma }\right] )\rightarrow M(B\left[ \tau _{\Gamma ^{\prime }}\right] )$.

\section{Group actions on pro-$C^{\ast }$-algebras}

Throughout this paper, $A\left[ \tau _{\Gamma }\right] $ is a pro-$C^{\ast }$%
-algebra with topology given by the family of $C^{\ast }$-seminorms $\Gamma
=\{p_{\lambda }\}_{\lambda \in \Lambda }$ and $G$ is a locally compact group.

\begin{definition}
\begin{enumerate}
\item \textit{An action} of $G$ on $A\left[ \tau _{\Gamma }\right] $ is a
group morphism $\alpha $ from $G$ to Aut$(A\left[ \tau _{\Gamma }\right] )$
such that the map $t\mapsto \alpha _{t}\left( a\right) $ from $G$ to $A\left[
\tau _{\Gamma }\right] $ is continuous for each $a\in A$.

\item An action $\alpha $ of $G$ on $A\left[ \tau _{\Gamma }\right] $ is 
\textit{strongly bounded}\textbf{,} if for each $\lambda \in \Lambda $ there
is $\mu \in \Lambda $ such that 
\begin{equation*}
p_{\lambda }\left( \alpha _{t}\left( a\right) \right) \leq p_{\mu }\left(
a\right)
\end{equation*}%
for all $t\in G$ and for all $a\in A$.

\item An action $\alpha $ is \textit{an inverse limit action}, if $%
p_{\lambda }\left( \alpha _{t}\left( a\right) \right) =p_{\lambda }\left(
a\right) $ for all $a\in A$, for all $t\in G$ and for all $\lambda \in
\Lambda $.
\end{enumerate}
\end{definition}

\begin{remark}
\begin{enumerate}
\item If $\alpha $ is an inverse limit action of $G$ on $A\left[ \tau
_{\Gamma }\right] $, then for each $\lambda \in \Lambda $,$\ $there is an
action $\alpha ^{\lambda }$ of $G\ $on $A_{\lambda }$ such that $\alpha
_{t}^{\lambda }\circ \pi _{\lambda }^{A}=\pi _{\lambda }^{A}\circ \alpha
_{t} $ for all $t\in G$, and then $\alpha _{t}=\lim\limits_{\leftharpoonup
\lambda }\alpha _{t}^{\lambda }$ for all $t\in G.$

\item Any inverse limit action of $G$ on $A\left[ \tau _{\Gamma }\right] $
is strongly bounded.

\item If $A$ is a $C^{\ast }$-algebra, then any action of $G$ on $A$ is
strongly bounded.

\item If $G$ is a compact group, then any action of $G$ on $A\left[ \tau
_{\Gamma }\right] $ is strongly bounded.
\end{enumerate}
\end{remark}

Let $X$ be a compactly countably generated Hausdorff topological space (that
is, $X$ is a direct limit of a countable family $\{K_{n}\}_{n}$ of compact
spaces). The $\ast $-algebra $C(X)$ of all continuous complex valued
functions on $X$ is a pro-$C^{\ast }$-algebra with topology given by the
family of $C^{\ast }$-seminorms $\{p_{K_{n}}\}_{n},$ where $p_{K_{n}}\left(
f\right) =\sup \{\left\vert f\left( x\right) \right\vert ;x\in K_{n}\}$.

\begin{example}
Let $(G,X)$ be a transformation group (that is, there is a continuous map $%
\left( t,x\right) $ $\mapsto t\cdot x$ from $G\times X$ to $X$ such that $%
e\cdot x=x$ and $s\cdot \left( t\cdot x\right) =\left( st\right) \cdot x$
for all $s,t\in G$ and for all $x\in X$) with $X=\lim\limits_{n\rightarrow
}K_{n}$ a compactly countably generated Hausdorff topological space. Then
there is an action $\alpha $ of $G$ on the pro-$C^{\ast }$-algebra $C(X)$,
given by 
\begin{equation*}
\alpha _{t}\left( f\right) \left( x\right) =f\left( t^{-1}\cdot x\right) 
\text{. }
\end{equation*}%
If for any positive integer $n$, there is a positive integer $m$ such that $%
G\cdot K_{n}\subseteq K_{m}$, the action $\alpha $ is strongly bounded,
since for each $n$, there is $m$ such that 
\begin{equation*}
p_{K_{n}}\left( \alpha _{t}\left( f\right) \right) =\sup \{\left\vert
f\left( t^{-1}\cdot x\right) \right\vert ;x\in K_{n}\}\leq \sup \{\left\vert
f\left( y\right) \right\vert ;y\in K_{m}\}=p_{K_{m}}\left( f\right)
\end{equation*}%
for all $f\in C(X)$ and for all $t\in G$. If $G\cdot K_{n}=K_{n}$ for all $n$%
, then $\alpha $ is an inverse limit action. Take, for instance, $\mathbb{R=}%
\lim\limits_{n\rightarrow }[-n,n]$. Suppose that $\mathbb{Z}_{2}$ actions on 
$\mathbb{R}$ by $\widehat{0}\cdot x$ $=x\ $and $\widehat{1}\cdot x$ $=2-x\ $%
for all $x\in \mathbb{R}$. Then $\left( \mathbb{Z}_{2},\mathbb{R}\right) $
is a transformation group such that for each positive integer $n$, $\mathbb{Z%
}_{2}\cdot \lbrack -n,n]\subseteq \lbrack -n-2,n+2]$.
\end{example}

\begin{example}
Let $X=\lim\limits_{n\rightarrow }K_{n}$ be a compactly countably generated
Hausdorff topological space and $h:X\rightarrow X$ a homeomorphism with the
property that for each positive integer $n$, there is a positive integer $m$
such that $h^{k}(K_{n})\subseteq K_{m}$ for all integers $k$. Then the map $%
n\mapsto \alpha _{n}\ $ from $\mathbb{Z}$ to Aut$\left( C(X)\right) $, where 
$\alpha _{n}(f)=f\circ h^{n}$, is a strong bounded action of $\mathbb{Z}$ to 
$C(X)$. If $h(K_{n})=K_{n}\ $ for all $n$, then $\alpha $ is an inverse
limit action. Take, for instance, $\mathbb{R=}\lim\limits_{n\rightarrow
}[-n,n]$, the map $h:\mathbb{R\rightarrow R}$ defined by $h(x)=1-x$ is a
homeomarphism such that for each positive integer $n$, $h^{k}([-n,n])%
\subseteq \lbrack -n-1,n+1]$ for all integers $k$.
\end{example}

\begin{example}
The $\ast $-algebra $C[0,1]$ equipped with the topology 'cc' of unifom
convergence on countable compact subsets is a pro-$C^{\ast }$-algebra
denoted by $C_{cc}[0,1]$ (see, for example, \cite[p. 104]{F}). The action of
\ $\mathbb{Z}_{2}$ on $C_{cc}[0,1]$ given by $\alpha _{\widehat{0}}=$id$%
_{C_{cc}[0,1]}$ and $\alpha _{\widehat{1}}\left( f\right) (x)=f(1-x)$ for
all $f\in C_{cc}[0,1]$ and for all $x\in \lbrack 0,1]$ is strongly bounded.
\end{example}

\begin{remark}
\begin{enumerate}
\item Let $\alpha $ be a strong bounded action of $G$ on $A\left[ \tau
_{\Gamma }\right] $. Then, for each $\lambda \in \Lambda $, the map $%
p^{\lambda }:A\rightarrow \lbrack 0,\infty )$ given by 
\begin{equation*}
p^{\lambda }\left( a\right) =\sup \{p_{\lambda }\left( \alpha _{t}\left(
a\right) \right) ;t\in G\}
\end{equation*}%
is a continuous $C^{\ast }$-seminorm on $A\left[ \tau _{\Gamma }\right] $.
Let $\Gamma ^{G}=\{p^{\lambda }\}_{\lambda \in \Lambda }$. Since, for each $%
\lambda \in \Lambda $, there is $\mu \in \Lambda $ such that 
\begin{equation*}
p_{\lambda }\leq p^{\lambda }\leq p_{\mu }\text{,}
\end{equation*}%
$\Gamma ^{G}$ defines on $A$ a structure of pro-$C^{\ast }$-algebra, and
moreover, the pro-$C^{\ast }$-algebras $A\left[ \tau _{\Gamma }\right] $ and 
$A\left[ \tau _{\Gamma }^{G}\right] $ are isomorphic.

\item If the action $\alpha $ of $G$ on $A\left[ \tau _{\Gamma }\right] $ is
strongly bounded, then $\alpha $ is an inverse limit action of $G$ on $A%
\left[ \tau _{\Gamma }^{G}\right] .$
\end{enumerate}
\end{remark}

\section{Covariant representations}

\begin{definition}
A pro-$C^{\ast }$-dynamical system is a triple $\left( G,\alpha ,A\left[
\tau _{\Gamma }\right] \right) $, where $G$ is a locally compact group, $A%
\left[ \tau _{\Gamma }\right] $ is a pro-$C^{\ast }$-algebra and $\alpha $
is an action of $G$ on $A\left[ \tau _{\Gamma }\right] $.
\end{definition}

\textit{A representation} of a pro-$C^{\ast }$-algebra $A\left[ {\small \tau 
}_{\Gamma }\right] $ on a Hilbert space $\mathcal{H}$ is a continuous $\ast $%
-morphism $\varphi :A\left[ {\small \tau }_{\Gamma }\right] \rightarrow L(%
\mathcal{H})$. A representation $\left( \varphi ,\mathcal{H}\right) $ of $A%
\left[ {\small \tau }_{\Gamma }\right] $ is\textit{\ nondegenerate} if $%
\left[ \varphi \left( A\right) \mathcal{H}\right] =\mathcal{H}$.

\begin{definition}
A covariant representation of $\left( G,\alpha ,A\left[ \tau _{\Gamma }%
\right] \right) $ on a Hilbert space $\mathcal{H}$ is a triple $\left(
\varphi ,u,\mathcal{H}\right) \ $consisting of a representation $\left(
\varphi ,\mathcal{H}\right) $ of $A\left[ \tau _{\Gamma }\right] $ on $%
\mathcal{H}$ and a unitary $\ast $-representation $\left( u,\mathcal{H}%
\right) $ of $G$ on $\mathcal{H}$ such that 
\begin{equation*}
\varphi \left( \alpha _{t}\left( a\right) \right) =u_{t}\varphi \left(
a\right) u_{t}^{\ast }
\end{equation*}%
for all $a\in A$ and for all $t\in G$. A covariant representation $\left(
\varphi ,u,\mathcal{H}\right) $ is nondegenerate if $\left( \varphi ,%
\mathcal{H}\right) $ is nondegenerate.

Two representations $\left( \varphi ,u,\mathcal{H}\right) $ and $\left( \psi
,v,\mathcal{K}\right) $ of $\left( G,\alpha ,A\left[ \tau _{\Gamma }\right]
\right) $ are unitarily equivalent if there is a unitary operator $U:%
\mathcal{H}\rightarrow \mathcal{K}$ such that $U\varphi \left( a\right)
=\psi \left( a\right) U$ for all $a\in A$ and $Uu_{t}=v_{t}U$ for all $t\in
G $.
\end{definition}

For each $p_{\lambda }\in \Gamma $, we denote by $\mathcal{R}_{\lambda
}\left( G,\alpha ,A\left[ \tau _{\Gamma }\right] \right) $ the collection of
all equivalence classes of nondegenerate covariant representations $\left(
\varphi ,u,\mathcal{H}\right) $ of $\left( G,\alpha ,A\left[ \tau _{\Gamma }%
\right] \right) $ with the property that $\left\Vert \varphi \left( a\right)
\right\Vert \leq p_{\lambda }\left( a\right) $ for all $a\in A$. Clearly, 
\begin{equation*}
\tbigcup\limits_{\lambda }\mathcal{R}_{\lambda }\left( G,\alpha ,A\left[
\tau _{\Gamma }\right] \right) =\mathcal{R}\left( G,\alpha ,A\left[ \tau
_{\Gamma }\right] \right) ,
\end{equation*}%
where $\mathcal{R}\left( G,\alpha ,A\left[ \tau _{\Gamma }\right] \right) $
denotes the collection of all equivalence classes of nondegenerate covariant
representations of $\left( G,\alpha ,A\left[ \tau _{\Gamma }\right] \right) $%
.

\begin{remark}
If $\alpha $ is an inverse limit action, then the map 
\begin{equation*}
\left( \varphi _{\lambda },u,\mathcal{H}\right) \rightarrow \left( \varphi
_{\lambda }\circ \pi _{\lambda }^{A},u,\mathcal{H}\right)
\end{equation*}%
is a bijection between $\mathcal{R}\left( G,\alpha ^{\lambda },A_{\lambda
}\right) $ and $\mathcal{R}_{\lambda }\left( G,\alpha ,A\left[ \tau _{\Gamma
}\right] \right) $ (see, \cite{J2}).
\end{remark}

By \cite{J2}, if $\alpha $ is an inverse limit action, then $\mathcal{R}%
\left( G,\alpha ,A\left[ \tau _{\Gamma }\right] \right) $ is non empty. From
this result and Remark 3.6, we conclude that if $\alpha $ is strongly
bounded, then $\mathcal{R}\left( G,\alpha ,A\left[ \tau _{\Gamma }\right]
\right) $ is non empty too. In the following proposition we give another
proof for this result.

\begin{proposition}
Let $\left( G,\alpha ,A\left[ \tau _{\Gamma }\right] \right) $ be a pro-$%
C^{\ast }$-dynamical system such that $\alpha $ is strongly bounded. Then
there is a covariant representation of $\left( G,\alpha ,A\left[ \tau
_{\Gamma }\right] \right) $.
\end{proposition}

\begin{proof}
Let $\left( \varphi ,\mathcal{H}\right) $ be a representation of $A\left[
\tau _{\Gamma }\right] $. Then there is $\lambda \in \Lambda $ such that $%
\left\Vert \varphi \left( a\right) \right\Vert \leq p_{\lambda }\left(
a\right) $ for all $a\in A$. Let $a\in A$ and $\xi \in L^{2}(G,\mathcal{H})$%
. Since, there is $p_{\mu }\in \Gamma $ such that 
\begin{eqnarray*}
\tint\limits_{G}\left\Vert \varphi \left( \alpha _{s^{-1}}\left( a\right)
\right) \left( \xi \left( s\right) \right) \right\Vert ^{2}ds &\leq
&\tint\limits_{G}\left\Vert \varphi \left( \alpha _{s^{-1}}\left( a\right)
\right) \right\Vert ^{2}\left\Vert \xi \left( s\right) \right\Vert ^{2}ds \\
&\leq &\tint\limits_{G}p_{\lambda }\left( \alpha _{s^{-1}}\left( a\right)
\right) ^{2}\left\Vert \xi \left( s\right) \right\Vert ^{2}ds\leq p_{\mu
}\left( a\right) ^{2}\left\Vert \xi \right\Vert ^{2},
\end{eqnarray*}%
the map $s\mapsto \varphi \left( \alpha _{s^{-1}}\left( a\right) \right)
\left( \xi \left( s\right) \right) $ defines an element in $L^{2}(G,\mathcal{%
H})$. Therefore, there is $\widetilde{\varphi }\left( a\right) \in L(L^{2}(G,%
\mathcal{H}))$ such that%
\begin{equation*}
\widetilde{\varphi }\left( a\right) \left( \xi \right) \left( s\right)
=\varphi \left( \alpha _{s^{-1}}\left( a\right) \right) \left( \xi \left(
s\right) \right) .
\end{equation*}%
In this way, we obtain a map $\widetilde{\varphi }:A\rightarrow L(L^{2}(G,%
\mathcal{H}))$. Moreover, $\widetilde{\varphi }$ is a continuous $\ast $%
-morphism, and then $\left( \widetilde{\varphi },L^{2}\left( G,\mathcal{H}%
\right) \right) $ is a representation of $A\left[ \tau _{\Gamma }\right] $.

Let $\left( \mathbf{\lambda }_{G}^{\mathcal{H}},L^{2}\left( G,\mathcal{H}%
\right) \right) $ be the unitary $\ast $-representation of $G$ on $%
L^{2}\left( G,\mathcal{H}\right) $ given by $\left( \mathbf{\lambda }_{G}^{%
\mathcal{H}}\right) _{t}\left( \xi \right) \left( s\right) =\xi \left(
t^{-1}s\right) $. It is easy to verify that $\left( \widetilde{\varphi },%
\mathbf{\lambda }_{G}^{\mathcal{H}},L^{2}\left( G,\mathcal{H}\right) \right) 
$ is a covariant representation of $\left( G,\alpha ,A\left[ \tau _{\Gamma }%
\right] \right) $.
\end{proof}

\begin{remark}
Let $\left( G,\alpha ,A\left[ \tau _{\Gamma }\right] \right) $ be a pro-$%
C^{\ast }$-dynamical system. Suppose that $\alpha $ is strongly bounded.
Then, for each representation $\left( \varphi ,\mathcal{H}\right) $ of $A%
\left[ \tau _{\Gamma }\right] $, $\ker \widetilde{\varphi }\subseteq \ker
\varphi $. Indeed, if $\ \widetilde{\varphi }\left( a\right) =0$, then $%
\varphi \left( \alpha _{s}\left( a\right) \right) \left( \xi \left( s\right)
\right) =0$ for all $s\in G$ and for all $\xi \in L^{2}\left( G,\mathcal{H}%
\right) $, whence $\varphi \left( a\right) \left( \xi \left( e\right)
\right) =0\ $for all $\xi \in L^{2}\left( G,\mathcal{H}\right) \ $and so $%
\varphi \left( a\right) =0$.
\end{remark}

\section{The full pro-$C^{\ast }$-crossed product}

Throughout this paper, $B\left[ \tau _{\Gamma ^{\prime }}\right] $ is a pro-$%
C^{\ast }$-algebra with topology given by the family of $C^{\ast }$%
-seminorms $\Gamma ^{\prime }=\{q_{\delta }\}_{\delta \in \Delta }$. Let $%
\left( G,\alpha ,A\left[ \tau _{\Gamma }\right] \right) $ be a pro-$C^{\ast
} $-dynamical system.

\begin{definition}
A covariant pro-$C^{\ast }$-morphism from $\left( G,\alpha ,A\left[ \tau
_{\Gamma }\right] \right) $ to a pro-$C^{\ast }$-algebra $B\left[ \tau
_{\Gamma ^{\prime }}\right] $ is a pair $\left( \varphi ,u\right) $
consisting of a pro-$C^{\ast }$-morphism $\varphi :A\left[ \tau _{\Gamma }%
\right] \rightarrow M(B\left[ \tau _{\Gamma ^{\prime }}\right] )$ and a
strict continuous group morphism $u:G\rightarrow \mathcal{U}(M(B\left[ \tau
_{\Gamma ^{\prime }}\right] ))$ such that 
\begin{equation*}
\varphi \left( \alpha _{t}\left( a\right) \right) =u_{t}\varphi \left(
a\right) u_{t}^{\ast }
\end{equation*}%
for all $t\in G$ and for all $a\in A$. A covariant pro-$C^{\ast }$-morphism $%
\left( \varphi ,u\right) $ from $\left( G,\alpha ,A\left[ \tau _{\Gamma }%
\right] \right) $ to $B\left[ \tau _{\Gamma ^{\prime }}\right] $ is
nondegenerate if $\left[ \varphi \left( A\right) B\right] =B\left[ \tau
_{\Gamma ^{\prime }}\right] $.
\end{definition}

\begin{theorem}
Let $\left( G,\alpha ,A\left[ \tau _{\Gamma }\right] \right) $ be a pro-$%
C^{\ast }$-dynamical system. If $\alpha $ is strongly bounded, then there is
a locally Hilbert space $\mathcal{H}$ and a covariant pro-$C^{\ast }$%
-morphism $\left( i_{A},i_{G}\right) $ from $\left( G,\alpha ,A\left[ \tau
_{\Gamma }\right] \right) $ to $L(\mathcal{H})$. Moreover, $i_{A}$ and $%
i_{G} $ are injective.
\end{theorem}

\begin{proof}
Let $\lambda \in \Lambda $. By Proposition 4.4, $\mathcal{R}_{\lambda
}\left( G,\alpha ,A\left[ \tau _{\Gamma }\right] \right) $ is non empty. Let 
$\left( \varphi ^{\lambda },u^{\lambda },H_{\lambda }\right) $ be the direct
sum of all equivalence classes of nondegenerate covariant reprsentations $%
\left( \varphi ,u,H_{\varphi ,u}\right) $ of $\left( G,\alpha ,A\left[ \tau
_{\Gamma }\right] \right) ,$ $\left( \varphi ,u,H_{\varphi ,u}\right) \in $ $%
\mathcal{R}_{\lambda }\left( G,\alpha ,A\left[ \tau _{\Gamma }\right]
\right) $. Then $\left( \varphi ^{\lambda },u^{\lambda },H_{\lambda }\right) 
$ is a nondegenerate covariant representation of $\left( G,\alpha ,A\left[
\tau _{\Gamma }\right] \right) $ such that $\left\Vert \varphi ^{\lambda
}\left( a\right) \right\Vert \leq p_{\lambda }\left( a\right) $\textbf{\ }%
for all\textbf{\ }$a\in A$\textbf{.}

Let $\mathcal{H}_{\lambda }=\oplus _{\mu \leq \lambda }H_{\mu }$. Then $%
\mathcal{H}=\lim\limits_{\lambda \rightarrow }\mathcal{H}_{\lambda }$ is a
locally Hilbert space. For $a\in A$, the map $i_{A}^{\lambda }\left(
a\right) :\mathcal{H}_{\lambda }\rightarrow \mathcal{H}_{\lambda }$ defined
by 
\begin{equation*}
i_{A}^{\lambda }\left( a\right) \left( \oplus _{\mu \leq \lambda }\xi _{\mu
}\right) =\oplus _{\mu \leq \lambda }\varphi ^{\mu }\left( a\right) \xi
_{\mu }
\end{equation*}%
is an element in $L(\mathcal{H}_{\lambda })$ and $\left\Vert i_{A}^{\lambda
}\left( a\right) \right\Vert \leq p_{\lambda }(a)$. Moreover, $%
i_{A}^{\lambda }\left( a^{\ast }\right) =i_{A}^{\lambda }\left( a\right)
^{\ast }$ and $i_{A}^{\lambda }\left( ab\right) =i_{A}^{\lambda }\left(
a\right) i_{A}^{\lambda }\left( b\right) $ for all $a,b\in A$. Clearly, $%
\left( i_{A}^{\lambda }\left( a\right) \right) _{\lambda }$ is a direct
system of bounded linear operators and $i_{A}\left( a\right)
=\lim\limits_{\lambda \rightarrow }i_{A}^{\lambda }\left( a\right) $ is an
element $L(\mathcal{H})\ $such that $i_{A}\left( a^{\ast }\right)
=i_{A}\left( a\right) ^{\ast }$ and $i_{A}\left( ab\right) =i_{A}\left(
a\right) i_{A}\left( b\right) $ for all $a,b\in A$. Moreover, 
\begin{equation*}
p_{\lambda ,L(\mathcal{H})}\left( i_{A}\left( a\right) \right) =\left\Vert
i_{A}^{\lambda }\left( a\right) \right\Vert \leq p_{\lambda }(a)
\end{equation*}%
for all $a\in A$ and for all $\lambda \in \Lambda $. Therefore, $i_{A}$ is a
pro-$C^{\ast }$-morphism.

For $t\in G$, the map $i_{G}^{\lambda }\left( t\right) :\mathcal{H}_{\lambda
}\rightarrow \mathcal{H}_{\lambda }$ defined by 
\begin{equation*}
i_{G}^{\lambda }\left( t\right) \left( \oplus _{\mu \leq \lambda }\xi _{\mu
}\right) =\oplus _{\mu \leq \lambda }u^{\mu }\left( t\right) \xi _{\mu }
\end{equation*}%
is a unitary element in $L(\mathcal{H}_{\lambda })$. Moreover, the map $%
t\mapsto i_{G}^{\lambda }\left( t\right) $ is a unitary $\ast $%
-representation of $G$ on $\mathcal{H}_{\lambda }$. Clearly, $\left(
i_{G}^{\lambda }\left( t\right) \right) _{\lambda }$ is a direct system of
unitary operators, and then $i_{G}\left( t\right) =\lim\limits_{\lambda
\rightarrow }i_{G}^{\lambda }\left( t\right) $ is a unitary element $L(%
\mathcal{H})$.$\ $Moreover, $t\mapsto i_{G}\left( t\right) $ is a group
morphism from $G$ to the group of unitary operators on $\mathcal{H}$, and
since for each $\xi \in \mathcal{H}$, the map $t\mapsto i_{G}\left( t\right)
\xi $ from $G$ to $\mathcal{H}$ is continuous, $t\mapsto i_{G}\left(
t\right) $ is a unitary $\ast $-representation of $G$ on $\mathcal{H}$. We
have%
\begin{eqnarray*}
i_{A}\left( \alpha _{t}\left( a\right) \right) \left( \oplus _{\mu \leq
\lambda }\xi _{\mu }\right) &=&i_{A}^{\lambda }\left( \alpha _{t}\left(
a\right) \right) \left( \oplus _{\mu \leq \lambda }\xi _{\mu }\right)
=\oplus _{\mu \leq \lambda }\varphi ^{\mu }\left( \alpha _{t}\left( a\right)
\right) \left( \xi _{\mu }\right) \\
&=&\oplus _{\mu \leq \lambda }u^{\mu }\left( t\right) \varphi ^{\mu }\left(
a\right) u^{\mu }\left( t\right) ^{\ast }\left( \xi _{\mu }\right) \\
&=&i_{G}(t)i_{A}\left( a\right) i_{G}(t)^{\ast }\left( \oplus _{\mu \leq
\lambda }\xi _{\mu }\right)
\end{eqnarray*}%
for all $a\in A$, for all $t\in G$ and for all $\oplus _{\mu \leq \lambda
}\xi _{\mu }\in \mathcal{H}_{\lambda },$ $\lambda \in \Lambda $, and so 
\begin{equation*}
i_{A}\left( \alpha _{t}\left( a\right) \right) =i_{G}(t)i_{A}\left( a\right)
i_{G}(t)^{\ast }
\end{equation*}%
for all $a\in A$ and for all $t\in G$.

Suppose that $i_{A}\left( a\right) =0$. Then $i_{A}^{\lambda }\left(
a\right) =0$\ for all $\lambda \in \Lambda $\ and so $\varphi \left(
a\right) =0$\ for all nondegenerate covariant representation $\left( \varphi
,u,H_{\varphi ,u}\right) $ of $\left( G,\alpha ,A\left[ \tau _{\Gamma }%
\right] \right) $. By Proposition 4.4 and Remark 4.5, $\psi \left( a\right)
=0$\ for all representations $\psi $\ of $A$. Therefore, $p_{\lambda }\left(
a\right) =0$\ for all $\lambda \in \Lambda $, and then $a=0$.

Suppose that $i_{G}\left( t\right) =$id$_{\mathcal{H}}$. Then $%
i_{G}^{\lambda }\left( t\right) =$id$_{\mathcal{H}_{\lambda }}$\ for all $%
\lambda \in \Lambda $,\ and so $u\left( t\right) =$id$_{H_{\varphi ,u}}$\
for all nondegenerate covariant representation $\left( \varphi ,u,H_{\varphi
,u}\right) $ of $\left( G,\alpha ,A\left[ \tau _{\Gamma }\right] \right) $,
whence we deduce that $t=e$.
\end{proof}

\begin{proposition}
Let $\left( G,\alpha ,A\left[ \tau _{\Gamma }\right] \right) $ be a pro-$%
C^{\ast }$-dynamical system. Then the following statements are equivalent.

\begin{enumerate}
\item $\alpha $ is an inverse limit action.

\item There is a locally Hilbert space $\mathcal{H}$ and a covariant pro-$%
C^{\ast }$-morphism $\left( i_{A},i_{G}\right) $\ from $\left( G,\alpha ,A%
\left[ \tau _{\Gamma }\right] \right) $ to $L(\mathcal{H})$ such that $%
p_{\lambda ,L\left( \mathcal{H}\right) }\left( i_{A}\left( a\right) \right)
=p_{\lambda }\left( a\right) $ for all $\lambda \in \Lambda $\textbf{\ } and 
$a\in A$.
\end{enumerate}
\end{proposition}

\begin{proof}
$(1)\Rightarrow (2)$ See \cite[Proposition 3.1]{J3} and \cite[Theorem 5.1]{I}%
.

$(2)\Rightarrow (1)$ From%
\begin{equation*}
i_{A}\left( \alpha _{t}\left( a\right) \right) =i_{G}\left( t\right)
i_{A}\left( a\right) i_{G}\left( t\right) ^{\ast }
\end{equation*}%
for all $t\in G$ and for all $a\in A$, and taking into account that $%
i_{G}\left( t\right) $ is a unitary element in $L(\mathcal{H})$ for all $%
t\in G$, we deduce that 
\begin{eqnarray*}
p_{\lambda }\left( \alpha _{t}\left( a\right) \right) &=&p_{\lambda ,L\left( 
\mathcal{H}\right) }\left( i_{A}\left( \alpha _{t}\left( a\right) \right)
\right) =p_{\lambda ,L\left( \mathcal{H}\right) }\left( i_{G}\left( t\right)
i_{A}\left( a\right) i_{G}\left( t\right) ^{\ast }\right) \\
&=&p_{\lambda ,L\left( \mathcal{H}\right) }\left( i_{A}\left( a\right)
\right) =p_{\lambda }(a)
\end{eqnarray*}%
for all $t\in G$, for all $a\in A$ and for all $t\in G$. Therefore, $\alpha $
is an inverse limit action.
\end{proof}

If $u$ is a strict continuous group morphism from $G$ to $\mathcal{U}\left(
M\left( B\left[ \tau _{\Gamma ^{\prime }}\right] \right) \right) $, then
there is a $\ast $-morphism $u:C_{c}\left( G\right) \rightarrow M\left( B%
\left[ \tau _{\Gamma ^{\prime }}\right] \right) $ given by $u\left( f\right)
=\tint\limits_{G}f(s)u_{s}ds,$ where $ds$ denotes the Haar measure on $G$
(sse \cite{J2})$.$

\begin{definition}
Let $\left( G,\alpha ,A\left[ \tau _{\Gamma }\right] \right) $ be a pro-$%
C^{\ast }$-dynamical system. A pro-$C^{\ast }$-algebra, denoted by $G\times
_{\alpha }A\left[ \tau _{\Gamma }\right] $,$\ $together with a covariant pro-%
$C^{\ast }$-morphism $\left( \iota _{A},\iota _{G}\right) $ from $\left(
G,\alpha ,A\left[ \tau _{\Gamma }\right] \right) $ to $G\times _{\alpha }A%
\left[ \tau _{\Gamma }\right] $ which verifies the following:

\begin{enumerate}
\item for each nondegenerate covariant representation $\left( \varphi ,u,%
\mathcal{H}\right) $ of $\left( G,\alpha ,A\left[ \tau _{\Gamma }\right]
\right) $, there is a nondegenerate representation $\left( \Phi ,\mathcal{H}%
\right) $ of $G\times _{\alpha }A\left[ \tau _{\Gamma }\right] $ such that $%
\overline{\Phi }\circ \iota _{A}=\varphi $ and $\overline{\Phi }\circ \iota
_{G}=u;$

\item $\overline{\text{span}\{\iota _{A}\left( a\right) \iota _{G}\left(
f\right) ;a\in A,f\in C_{c}\left( G\right) \}}=G\times _{\alpha }A\left[
\tau _{\Gamma }\right] ;$
\end{enumerate}

is called the full pro-$C^{\ast }$-crossed product of $A\left[ \tau _{\Gamma
}\right] $ by $\alpha $.
\end{definition}

\begin{remark}
The covariant morphism $\left( \iota _{A},\iota _{G}\right) $ from the above
definition is nondegenerate.
\end{remark}

\begin{proposition}
Let $\left( G,\alpha ,A\left[ \tau _{\Gamma }\right] \right) $ be a pro-$%
C^{\ast }$-dynamical system such that there is a full pro-$C^{\ast }$%
-crossed product of $A\left[ \tau _{\Gamma }\right] $ by $\alpha $ and $%
\left( \varphi ,u\right) $ a nondegenerate covariant morphism from $\left(
G,\alpha ,A\left[ \tau _{\Gamma }\right] \right) $ to a pro-$C^{\ast }$%
-algebra $B\left[ \tau _{\Gamma ^{\prime }}\right] $. Then there is a unique
nondegenerate pro-$C^{\ast }$-morphism $\varphi \times u:G\times _{\alpha }A%
\left[ \tau _{\Gamma }\right] \rightarrow M(B\left[ \tau _{\Gamma ^{\prime }}%
\right] )$ such that 
\begin{equation*}
\overline{\varphi \times u}\circ \iota _{A}=\varphi \text{ and }\overline{%
\varphi \times u}\circ \iota _{G}=u.
\end{equation*}%
Moreover, the map $\left( \varphi ,u\right) \rightarrow \varphi \times u$ is
a bijection between nondegenerate covariant morphisms of $\left( G,\alpha ,A%
\left[ \tau _{\Gamma }\right] \right) $ onto nondegenerate morphisms of $%
G\times _{\alpha }A\left[ \tau _{\Gamma }\right] $.
\end{proposition}

\begin{proof}
Let $q_{\delta }\in \Gamma ^{\prime }$ and $\left( \psi _{\delta },\mathcal{H%
}\right) $ a faithful nondegenerate representation of $B_{\delta }$. Then, $(%
\overline{\psi _{\delta }}\circ \overline{\pi _{\delta }^{B}}\circ \varphi ,$
$\overline{\psi _{\delta }}\circ \overline{\pi _{\delta }^{B}}\circ u,%
\mathcal{H)}$ is a nondegenerate covariant representation of $\left(
G,\alpha ,A\left[ \tau _{\Gamma }\right] \right) $, and by Definition 5.4,
there is a nondegenerate representation $\left( \phi _{\delta },\mathcal{H}%
\right) $ of $G\times _{\alpha }A\left[ \tau _{\Gamma }\right] $ such that 
\begin{equation*}
\overline{\phi _{\delta }}\circ \iota _{A}=\overline{\psi _{\delta }}\circ 
\overline{\pi _{\delta }^{B}}\circ \varphi \text{ and }\overline{\phi
_{\delta }}\circ \iota _{G}=\overline{\psi _{\delta }}\circ \overline{\pi
_{\delta }^{B}}\circ u.
\end{equation*}%
Let $\Phi _{\delta }=\overline{\psi _{\delta }^{-1}}\circ \phi _{\delta }$.
Then $\Phi _{\delta }$ is a nondegenerate pro-$C^{\ast }$-morphism from $%
G\times _{\alpha }A\left[ \tau _{\Gamma }\right] $ to $M(B_{\delta })$.
Moreover, for $q_{\delta _{1}},q_{\delta _{2}}\in \Gamma ^{\prime }$ with $%
q_{\delta _{1}}\geq q_{\delta _{2}}$, we have $\overline{\pi _{\delta
_{1}\delta _{2}}^{B}}\circ \Phi _{\delta _{1}}=\Phi _{_{\delta _{2}}}$.
Therefore, there is a nondegenerate pro-$C^{\ast }$-morphism $\varphi \times
u:G\times _{\alpha }A\left[ \tau _{\Gamma }\right] \rightarrow M(B\left[
\tau _{\Gamma ^{\prime }}\right] )$ such that 
\begin{equation*}
\overline{\pi _{\delta }^{B}}\circ \varphi \times u=\Phi _{\delta }
\end{equation*}%
for all $q_{\delta }\in \Gamma ^{\prime }$. Moreover, $\overline{\varphi
\times u}\circ \iota _{A}=\varphi $ and $\overline{\varphi \times u}\circ
\iota _{G}=u$, and since $\{\iota _{A}\left( a\right) \iota _{G}\left(
f\right) ,$ $a\in A,$ $f\in C_{c}\left( G\right) \}$ generates $G\times
_{\alpha }A\left[ \tau _{\Gamma }\right] $, $\varphi \times u$ is unique
with the above properties.

Let $\Phi :G\times _{\alpha }A\left[ \tau _{\Gamma }\right] \rightarrow M($ $%
B\left[ \tau _{\Gamma ^{\prime }}\right] )$ be a nondegenerate pro-$C^{\ast
} $-morphism. Then $\varphi =$ $\overline{\Phi }\circ \iota _{A}$ is a
nondegenerate pro-$C^{\ast }$-morphism from $A\left[ \tau _{\Gamma }\right] $
to $M(B\left[ \tau _{\Gamma ^{\prime }}\right] )$ and $u=$ $\overline{\Phi }%
\circ \iota _{G}$ is a strict continuous morphism from $G$ to $\mathcal{U}%
(M(B\left[ \tau _{\Gamma ^{\prime }}\right] ))$, since $\iota _{G}$ is a
strict continuous morphism from $G$ to $M(G\times _{\alpha }A\left[ \tau
_{\Gamma }\right] )$ and $\overline{\Phi }$ is strongly continuous on the
bounded subsets of $M(G\times _{\alpha }A\left[ \tau _{\Gamma }\right] )$.
Moreover, $\left( \varphi ,u\right) $ is a nondegenerate covariant morphism
from $A\left[ \tau _{\Gamma }\right] $ to $B\left[ \tau _{\Gamma ^{\prime }}%
\right] $, and $\varphi \times u=\Phi $. If $\left( \psi ,v\right) $ is
another nondegenerate covariant morphism from $A\left[ \tau _{\Gamma }\right]
$ to $B\left[ \tau _{\Gamma ^{\prime }}\right] $ such that $\psi \times
v=\Phi $, then $\psi =$ $\overline{\Phi }\circ \iota _{A}=\varphi $ and $v=$ 
$\overline{\Phi }\circ \iota _{G}=u$.
\end{proof}

The following corollary provides uniqueness of the full pro-$C^{\ast }$%
-crossed product by strong bounded actions.

\begin{corollary}
Let $\left( G,\alpha ,A\left[ \tau _{\Gamma }\right] \right) $ be a pro-$%
C^{\ast }$-dynamical system such that there is a full pro-$C^{\ast }$%
-crossed product of $A\left[ \tau _{\Gamma }\right] $ by $\alpha $\textbf{. }%
Then the full pro-$C^{\ast }$-crossed product of $A\left[ \tau _{\Gamma }%
\right] $ by $\alpha $ is unique up to a pro-$C^{\ast }$-isomorphism.
\end{corollary}

\begin{proof}
Let $B\left[ \tau _{\Gamma ^{\prime }}\right] $ be a pro-$C^{\ast }$-algebra
and $\left( j_{A},j_{G}\right) $ a covariant pro-$C^{\ast }$-morphism from $%
\left( G,\alpha ,A\left[ \tau _{\Gamma }\right] \right) $ to $B\left[ \tau
_{\Gamma ^{\prime }}\right] $ which satisfy the relations $(1)-(2)$ from
Definition 5.4. Then, by Proposition 5.6, there is a nondegenerate pro-$%
C^{\ast }$-morphism $\Phi :G\times _{\alpha }A\left[ \tau _{\Gamma }\right]
\rightarrow M(B\left[ \tau _{\Gamma ^{\prime }}\right] )$ such that $%
\overline{\Phi }\circ \iota _{A}=j_{A}$ and $\overline{\Phi }\circ \iota
_{G}=j_{G}$. Since $\{\iota _{A}\left( a\right) \iota _{G}\left( f\right) ,$ 
$a\in A,$ $f\in C_{c}\left( G\right) \}$ generates $G\times _{\alpha }A\left[
\tau _{\Gamma }\right] $ and $\{j_{A}\left( a\right) j_{G}\left( f\right) ,$ 
$a\in A,$ $f\in C_{c}\left( G\right) \}$ generates $B\left[ \tau _{\Gamma
^{\prime }}\right] ,$ $\Phi \left( G\times _{\alpha }A\left[ \tau _{\Gamma }%
\right] \right) \subseteq B.$

In the same way, there is a pro-$C^{\ast }$-morphism $\Psi :B\left[ \tau
_{\Gamma ^{\prime }}\right] \rightarrow G\times _{\alpha }A\left[ \tau
_{\Gamma }\right] $ such that $\overline{\Psi }\circ j_{A}=\iota _{A}$ and $%
\overline{\Psi }\circ j_{G}=\iota _{G}$. From these facts and Definition
5.4(2)\textbf{,} we deduce that $\Phi \circ \Psi =$id$_{B}$ and $\Psi \circ
\Phi =$id$_{G\times _{\alpha }A\left[ \tau _{\Gamma }\right] }$, and so $%
\Phi $ is a pro-$C^{\ast }$-isomorphism.
\end{proof}

The following proposition relates the nondegenerate covariant
representations of a pro-$C^{\ast }$-dynamical system $\left( G,\alpha ,A%
\left[ \tau _{\Gamma }\right] \right) $ with the nondegenerate
representations of the full pro-$C^{\ast }$-crossed product of $A\left[ \tau
_{\Gamma }\right] $ by $\alpha $\textbf{.}

\begin{proposition}
Let $\left( G,\alpha ,A\left[ \tau _{\Gamma }\right] \right) $ be a pro-$%
C^{\ast }$-dynamical system such that there is the full pro-$C^{\ast }$%
-crossed product of $A\left[ \tau _{\Gamma }\right] $ by $\alpha $\textbf{. }%
There is a bijective correspondence between nondegenerate covariant
representations of $\left( G,\alpha ,A\left[ \tau _{\Gamma }\right] \right) $
and nondegenerate representations of $G\times _{\alpha }A\left[ \tau
_{\Gamma }\right] $.
\end{proposition}

\begin{proof}
Let $\left( \varphi ,u,\mathcal{H}\right) $ be a nondegenerate covariant
representation of $\left( G,\alpha ,A\left[ \tau _{\Gamma }\right] \right) $%
. Then, by Definition 5.4, there is a representation $\left( \varphi \times
u,\mathcal{H}\right) $ such that $\overline{\varphi \times u}\circ \iota
_{A}=\varphi \ $and $\overline{\varphi \times u}\circ \iota _{G}=u$.
Moreover, by Definition 5.4(2)\textbf{,} $\left( \varphi \times u,\mathcal{H}%
\right) $ is unique, and since $\varphi $ is nondegenerate, it is
nondegenerate too.

Let $\left( \Phi ,\mathcal{H}\right) $ be a nondegenerate representation of $%
G\times _{\alpha }A\left[ \tau _{\Gamma }\right] $. Then $(\overline{\Phi }%
\circ \iota _{A},$ $\overline{\Phi }\circ \iota _{G},\mathcal{H)}$ is a
covariant representation of $\left( G,\alpha ,A\left[ \tau _{\Gamma }\right]
\right) $, and moreover, $\left( \overline{\Phi }\circ \iota _{A}\right)
\times \left( \overline{\Phi }\circ \iota _{G}\right) =\Phi $. Since $\iota
_{A}$ and $\Phi $ are nondegenerate, the net $\{\overline{\Phi }\left( \iota
_{A}\left( e_{i}\right) \right) \}_{i}$, where $\{e_{i}\}_{i}$ is an
approximate unit of $A\left[ \tau _{\Gamma }\right] $, converges strictly to
id$_{\mathcal{H}}$, and so $\overline{\Phi }\circ \iota _{A}$ is
nondegenerate.

Suppose that there is another nondegenerate covariant representation $\left(
\varphi ,u,\mathcal{H}\right) $ of $\left( G,\alpha ,A\left[ \tau _{\Gamma }%
\right] \right) $ such that $\varphi \times u=\Phi $. Then $\varphi =%
\overline{\varphi \times u}\circ \iota _{A}=\overline{\Phi }\circ \iota _{A}$
and $u=\overline{\varphi \times }u\circ \iota _{G}=\overline{\Phi }\circ
\iota _{G}$. Therefore, the map $\left( \varphi ,u,\mathcal{H}\right)
\mapsto \left( \varphi \times u,\mathcal{H}\right) $ is bijective.
\end{proof}

\begin{theorem}
Let $\left( G,\alpha ,A\left[ \tau _{\Gamma }\right] \right) $ be a pro-$%
C^{\ast }$-dynamical system such that $\alpha $ is strongly bounded. Then,
there is the full pro-$C^{\ast }$-crossed product of $A\left[ \tau _{\Gamma }%
\right] $ by $\alpha $.
\end{theorem}

\begin{proof}
By Theorem 5.2\textbf{,} there is a locally Hilbert space $\mathcal{H}$ and
a covariant pro-$C^{\ast }$-morphism $\left( i_{A},i_{G}\right) $ from $%
A[\tau _{\Gamma }]$ to $L(\mathcal{H})$.

Let $B=\overline{\text{span}\{i_{A}\left( a\right) i_{G}\left( f\right)
;a\in A,f\in C_{c}\left( G\right) \}}\subseteq L(\mathcal{H})$. To show that 
$B$ is a pro-$C^{\ast }$-algebra, we must show that $B$ is closed under
taking adjoints and multiplication. For this, since $B=\lim\limits_{%
\longleftarrow \lambda }\overline{\pi _{\lambda }^{\mathcal{H}}\left(
B\right) }$ (\cite[Chapter III, Theorem 3.1]{M}), it is sufficient to show
that for each $p_{\lambda }\in \Gamma $, $\pi _{\lambda }^{\mathcal{H}%
}\left( i_{G}\left( f\right) i_{A}\left( a\right) \right) $ and $\pi
_{\lambda }^{\mathcal{H}}\left( i_{A}\left( b\right) i_{G}\left( f\right)
i_{A}\left( a\right) i_{G}(h\right) )$ are elements in the closure of $\pi
_{\lambda }^{\mathcal{H}}\left( B\right) $ in $L\left( \mathcal{H}_{\lambda
}\right) $ for all $a,b\in A\ $and for all $f,h\in C_{c}(G)$.

The map $s\rightarrow \pi _{\lambda }^{A}\left( f\left( s\right) \alpha
_{s}(a)\right) $ from $G$ to $A_{\lambda }$ defines an element in $%
C_{c}(G,A_{\lambda })$, and so there is a net $\{\pi _{\lambda
}^{A}(a_{j})\otimes f_{j}\}_{j\in J}$ in $A_{\lambda }\otimes _{\text{alg}%
}C_{c}(G)$ with \textbf{\ }supp$f_{j},\ $supp$f\subseteq K$ for some compact
subset $K$, which converges uniformly to this map.

By \cite[Lemma 3.7]{J4},%
\begin{eqnarray*}
\pi _{\lambda }^{\mathcal{H}}\left( i_{G}\left( f\right) i_{A}\left(
a\right) \right) &=&\tint\limits_{G}f\left( s\right) i^{\lambda }\left(
s\right) ds\pi _{\lambda }^{\mathcal{H}}\left( i_{A}(a)\right)
=\tint\limits_{G}f\left( s\right) \pi _{\lambda }^{\mathcal{H}}\left(
i_{G}(s)i_{A}(a)\right) ds \\
&=&\tint\limits_{G}f\left( s\right) \pi _{\lambda }^{\mathcal{H}}\left(
i_{A}(\alpha _{s}\left( a\right) )i_{G}(s)\right) ds \\
&=&\tint\limits_{G}f\left( s\right) i_{A}^{\lambda }(\alpha
_{s}(a))i_{G}^{\lambda }(s)ds
\end{eqnarray*}%
and%
\begin{equation*}
\pi _{\lambda }^{\mathcal{H}}\left( i_{A}\left( a_{j}\right) i_{G}\left(
f_{j}\right) \right) =\pi _{\lambda }^{\mathcal{H}}\left( i_{A}\left(
a_{j}\right) \right) \tint\limits_{G}f_{j}\left( s\right) i_{G}^{\lambda
}(s)ds=\tint\limits_{G}i_{A}^{\lambda }(a_{j})f_{j}\left( s\right)
i_{G}^{\lambda }(s)ds
\end{equation*}%
for each $j\in J$. Then, we have 
\begin{eqnarray*}
&&\left\Vert \pi _{\lambda }^{\mathcal{H}}\left( i_{G}\left( f\right)
i_{A}\left( a\right) \right) -\pi _{\lambda }^{\mathcal{H}}\left(
i_{A}\left( a_{j}\right) i_{G}\left( f_{j}\right) \right) \right\Vert
_{L\left( \mathcal{H}_{\lambda }\right) } \\
&\leq &\tint\limits_{G}\left\Vert f\left( s\right) i_{A}^{\lambda }(\alpha
_{s}(a))i_{G}^{\lambda }(s)-i_{A}^{\lambda }(a_{j})f_{j}\left( s\right)
i_{G}^{\lambda }(s)\right\Vert _{L\left( \mathcal{H}_{\lambda }\right) }ds \\
&\leq &M\sup \{\left\Vert f\left( s\right) i_{A}^{\lambda }(\alpha
_{s}(a))i_{G}^{\lambda }(s)-i_{A}^{\lambda }(a_{j})f_{j}\left( s\right)
i_{G}^{\lambda }(s)\right\Vert _{L\left( \mathcal{H}_{\lambda }\right)
},s\in K\} \\
&=&M\sup \{\left\Vert i_{A}^{\lambda }(f\left( s\right) \alpha
_{s}(a)-f_{j}\left( s\right) a_{j})\right\Vert _{L\left( \mathcal{H}%
_{\lambda }\right) }\left\Vert i_{G}^{\lambda }(s)\right\Vert _{L\left( 
\mathcal{H}_{\lambda }\right) },s\in K\} \\
&\leq &M\sup \{p_{\lambda }(f\left( s\right) \alpha _{s}(a)-f_{j}\left(
s\right) a_{j}),s\in K\} \\
&=&M\sup \{\left\Vert \pi _{\lambda }^{A}(f\left( s\right) \alpha
_{s}(a))-f_{j}\left( s\right) \pi _{\lambda }^{A}\left( a_{j}\right)
\right\Vert _{A_{\lambda }},s\in K\}
\end{eqnarray*}%
for all $j\in J$, where $M=\tint\limits_{K}dg$, and so $\pi _{\lambda }^{%
\mathcal{H}}\left( i_{G}\left( f\right) i_{A}\left( a\right) \right) \in 
\overline{\pi _{\lambda }^{\mathcal{H}}(B)}$.

On the other hand, 
\begin{eqnarray*}
&&\left\Vert \pi _{\lambda }^{\mathcal{H}}\left( i_{A}\left( b\right)
i_{G}\left( f\right) i_{A}\left( a\right) i_{G}(h)\right) -\pi _{\lambda }^{%
\mathcal{H}}\left( i_{A}\left( ba_{j}\right) i_{G}\left( f_{j}\ast h\right)
\right) \right\Vert _{L\left( \mathcal{H}_{\lambda }\right) } \\
&=&\left\Vert \pi _{\lambda }^{\mathcal{H}}\left( i_{A}\left( b\right)
i_{G}\left( f\right) i_{A}\left( a\right) i_{G}(h)-i_{A}(b)i_{A}\left(
a_{j}\right) i_{G}\left( f_{j}\right) i_{G}(h)\right) \right\Vert _{L\left( 
\mathcal{H}_{\lambda }\right) } \\
&\leq &\left\Vert \pi _{\lambda }^{\mathcal{H}}\left( i_{A}\left( b\right)
\right) \pi _{\lambda }^{\mathcal{H}}\left( i_{G}\left( f\right) i_{A}\left(
a\right) -i_{A}\left( a_{j}\right) i_{G}\left( f_{j}\right) \right) \pi
_{\lambda }^{\mathcal{H}}\left( i_{G}\left( h\right) \right) \right\Vert
_{L\left( \mathcal{H}_{\lambda }\right) } \\
&\leq &\left\Vert \pi _{\lambda }^{\mathcal{H}}\left( i_{A}\left( b\right)
\right) \right\Vert _{L\left( \mathcal{H}_{\lambda }\right) }\left\Vert \pi
_{\lambda }^{\mathcal{H}}\left( i_{G}\left( h\right) \right) \right\Vert
_{L\left( \mathcal{H}_{\lambda }\right) }\left\Vert \pi _{\lambda }^{%
\mathcal{H}}\left( i_{G}\left( f\right) i_{A}\left( a\right) \right) -\pi
_{\lambda }^{\mathcal{H}}\left( i_{A}\left( a_{j}\right) i_{G}\left(
f_{j}\right) \right) \right\Vert _{L\left( \mathcal{H}_{\lambda }\right) }
\end{eqnarray*}%
whence, we deduce that $\pi _{\lambda }^{\mathcal{H}}\left( i_{A}\left(
b\right) i_{G}\left( f\right) i_{A}\left( a\right) i_{G}(h)\right) \in 
\overline{\pi _{\lambda }^{\mathcal{H}}(B)}$. Thus, we showed that $\pi
_{\lambda }^{\mathcal{H}}\left( i_{G}\left( f\right) i_{A}\left( a\right)
\right) ,\pi _{\lambda }^{\mathcal{H}}\left( i_{A}\left( b\right)
i_{G}\left( f\right) i_{A}\left( a\right) i_{G}(h)\right) \in \overline{\pi
_{\lambda }^{\mathcal{H}}(B)}$ for each $\lambda \in \Lambda $, and so $%
i_{G}\left( f\right) i_{A}\left( a\right) $,$\ i_{A}\left( b\right)
i_{G}\left( f\right) i_{A}\left( a\right) i_{G}(h)\in B$. Therefore, $B$ is
a pro-$C^{\ast }$-algebra.

In the same manner, we show that for each $a\in A,i_{A}\left( a\right)
i_{A}\left( b\right) i_{G}(f)\in B$ and $i_{A}\left( b\right)
i_{G}(f)i_{A}\left( a\right) \in B$ for all $b\in A$ and for all $f\in
C_{c}(G)$, and so $i_{A}\left( a\right) \in M(B)$.

From,%
\begin{equation*}
i_{G}(t)i_{A}\left( a\right) i_{G}(f)=\tint\limits_{G}f(s)i_{A}\left( \alpha
_{t}(a)\right) i_{G}(ts)ds\in B
\end{equation*}%
and%
\begin{equation*}
i_{G}(f)i_{A}\left( a\right) i_{G}(t)=\tint\limits_{G}f(s)i_{A}\left( \alpha
_{s}\left( a\right) \right) i_{G}(st)ds\in B
\end{equation*}%
for all $a\in A$, for all $f\in C_{c}(G)$ and for all $t\in G$, we deduce
that $i_{G}(t)\in M(B)$ for all $t\in G$.

Let $\left( \psi ,v,H_{\psi ,v}\right) $ be a nondegenerate covariant
representation of $\left( G,\alpha ,A\left[ \tau _{\Gamma }\right] \right) $%
. Then there is $\left( \varphi ,u,H_{\varphi ,u}\right) \in \mathcal{R}%
_{\lambda }\left( G,\alpha ,A\left[ \tau _{\Gamma }\right] \right) $ such
that $\left( \psi ,v,H_{\psi ,v}\right) $ and $\left( \varphi ,u,H_{\varphi
,u}\right) $ are unitarily equivalent. So there is a unitary operator $U:$ $%
H_{\psi ,v}\rightarrow H_{\varphi ,u}$ such that $\psi \left( a\right)
=U^{\ast }\varphi (a)U$ for all $a\in A$\ and $v_{t}=U^{\ast }u_{t}U$ for
all $t\in G$. The map $\Psi :L(\mathcal{H})\rightarrow L(H_{\lambda })$
given by 
\begin{equation*}
\Psi \left( T\right) =\pi _{\lambda }^{\mathcal{H}}\left( T\right)
|_{H_{\lambda }}
\end{equation*}%
is a representation of $L(\mathcal{H})$ on $H_{\lambda }$ (see the proof of
Theorem 5.2). From 
\begin{equation*}
\Psi \left( i_{A}(a)\right) \left( H_{\varphi ,u}\right) =i_{A}^{\lambda
}(a)\left( H_{\varphi ,u}\right) \subseteq H_{\varphi ,u}
\end{equation*}%
for all $a\in A$ and 
\begin{equation*}
\Psi \left( i_{G}(t)\right) \left( H_{\varphi ,u}\right) =i_{G}^{\lambda
}(t)\left( H_{\varphi ,u}\right) \subseteq H_{\varphi ,u}
\end{equation*}%
for all $t\in G$, and taking into account that $B$ is generated by $%
\{i_{A}\left( a\right) i_{G}\left( f\right) ;a\in A,f\in C_{c}\left(
G\right) \}$, we deduce that $\Psi \left( B\right) \left( H_{\varphi
,u}\right) \subseteq H_{\varphi ,u}$. Let $\Phi :B\rightarrow $ $L(H_{\psi
,v})$ given by 
\begin{equation*}
\Phi \left( b\right) =U^{\ast }\Psi \left( b\right) |_{H_{\varphi ,u}}U.
\end{equation*}%
Clearly, $\Phi $ is a nondegenerate representation of $B$ on $H_{\psi ,v}$, 
\begin{equation*}
\overline{\Phi }\left( i_{A}(a)\right) =U^{\ast }\Psi \left( i_{A}(a)\right)
|_{H_{\varphi ,u}}U=U^{\ast }i_{A}^{\lambda }(a)|_{H_{\varphi ,u}}U=U^{\ast
}\varphi (a)U=\psi \left( a\right)
\end{equation*}%
for all $a\in A$, and 
\begin{equation*}
\overline{\Phi }\left( i_{G}(t)\right) =U^{\ast }\Psi \left( i_{G}(t)\right)
|_{H_{\varphi ,u}}U=U^{\ast }i_{G}^{\lambda }(t)|_{H_{\varphi ,u}}U=U^{\ast
}u_{t}U=v_{t}
\end{equation*}%
for all $t\in G$.
\end{proof}

\begin{remark}
The index of the family of seminorms which gives the topology on the full
pro-$C^{\ast }$-crossed product of $A\left[ \tau _{\Gamma }\right] $ by $%
\alpha $ is the same with the index of the family of seminorms which gives
the topology on $A\left[ \tau _{\Gamma }\right] .$
\end{remark}

\begin{proposition}
Let $\left( G,\alpha ,A\left[ \tau _{\Gamma }\right] \right) $ be a pro-$%
C^{\ast }$-dynamical system such that $\alpha $ is an inverse limit action.%
\textbf{\ }Then for each $\lambda \in \Lambda $, the $C^{\ast }$-algebra $%
\left( G\times _{\alpha }A\left[ \tau _{\Gamma }\right] \right) _{\lambda }$
is isomorphic to the full $C^{\ast }$-crossed product of $A_{\lambda }$ by $%
\alpha ^{\lambda }$.
\end{proposition}

\begin{proof}
By Theorem 5.2, Proposition 5.3 and Corollary 5.7, there is a $C^{\ast }$%
-morphism $i_{A_{\lambda }}:A_{\lambda }\rightarrow M\left( \left( G\times
_{\alpha }A\left[ \tau _{\Gamma }\right] \right) _{\lambda }\right) $ such
that $i_{A_{\lambda }}\circ \pi _{\lambda }^{A}=\overline{\pi _{\lambda
}^{G\times _{\alpha }A\left[ \tau _{\Gamma }\right] }}\circ i_{A}$. Using
the fact that $\alpha $ is an inverse limit action,\textbf{\ }it is easy to
check that $\left( i_{A_{\lambda }},\overline{\pi _{\lambda }^{G\times
_{\alpha }A\left[ \tau _{\Gamma }\right] }}\circ i_{G}\right) $ is a
covariant $C^{\ast }$-morphism from $\left( G,\alpha ^{\lambda },A_{\lambda
}\right) $ to $\left( G\times _{\alpha }A\left[ \tau _{\Gamma }\right]
\right) _{\lambda }$. Moreover, 
\begin{eqnarray*}
&&\overline{\text{span}\{i_{A_{\lambda }}\left( \pi _{\lambda }^{A}\left(
a\right) \right) \overline{\pi _{\lambda }^{G\times _{\alpha }A\left[ \tau
_{\Gamma }\right] }}\left( i_{G}\left( f\right) \right) ;a\in A,f\in
C_{c}\left( G\right) \}} \\
&=&\overline{\text{span}\{\overline{\pi _{\lambda }^{G\times _{\alpha }A%
\left[ \tau _{\Gamma }\right] }}\left( i_{A}\left( a\right) i_{G}\left(
f\right) \right) ;a\in A,f\in C_{c}\left( G\right) \}} \\
&=&\overline{\pi _{\lambda }^{G\times _{\alpha }A\left[ \tau _{\Gamma }%
\right] }}\left( G\times _{\alpha }A\left[ \tau _{\Gamma }\right] \right)
=\left( G\times _{\alpha }A\left[ \tau _{\Gamma }\right] \right) _{\lambda }.
\end{eqnarray*}%
Let $\left( \varphi ,u,\mathcal{H}\right) $ be a nondegenerate covariant
representation of $\left( G,\alpha ^{\lambda },A_{\lambda }\right) $. Then $%
\left( \varphi \circ \pi _{\lambda }^{A},u,\mathcal{H}\right) $ is a
nondegenerate covariant representation of $\left( G,\alpha ,A\left[ \tau
_{\Gamma }\right] \right) $ and by Definition\textbf{\ }5.4\textbf{, }there
is a nondegenerate representation $\left( \Phi ,\mathcal{H}\right) $ of $%
G\times _{\alpha }A\left[ \tau _{\Gamma }\right] $ such that $\overline{\Phi 
}\circ i_{A}=\varphi \circ \pi _{\lambda }^{A}$ and $\overline{\Phi }\circ
i_{G}=u.$ Moreover, by the proof of Theorem 5.9\textbf{, }%
\begin{equation*}
\left\Vert \Phi \left( b\right) \right\Vert \leq p_{\lambda ,G\times
_{\alpha }A\left[ \tau _{\Gamma }\right] }\left( b\right)
\end{equation*}%
for all $b\in G\times _{\alpha }A\left[ \tau _{\Gamma }\right] $. Therefore,
there is the $C^{\ast }$-morphism $\Phi _{\lambda }:\left( G\times _{\alpha
}A\left[ \tau _{\Gamma }\right] \right) _{\lambda }\rightarrow L(\mathcal{H}%
) $ such that $\Phi _{\lambda }\circ \overline{\pi _{\lambda }^{G\times
_{\alpha }A\left[ \tau _{\Gamma }\right] }}=\Phi $. Moreover, $\left( \Phi
_{\lambda },\mathcal{H}\right) $ is a nondegenerate representation of $%
\left( G\times _{\alpha }A\left[ \tau _{\Gamma }\right] \right) _{\lambda }\ 
$such that 
\begin{equation*}
\overline{\Phi }_{\lambda }\circ i_{A_{\lambda }}=\varphi \text{ and }%
\overline{\Phi _{\lambda }}\circ \left( \overline{\pi _{\lambda }^{G\times
_{\alpha }A\left[ \tau _{\Gamma }\right] }}\circ i_{G}\right) =u.
\end{equation*}%
Thus, we showed that $\left( G\times _{\alpha }A\left[ \tau _{\Gamma }\right]
\right) _{\lambda }$ is isomorphic to $G\times _{\alpha ^{\lambda
}}A_{\lambda }.$
\end{proof}

\begin{corollary}
Let $\left( G,\alpha ,A\left[ \tau _{\Gamma }\right] \right) $ be a pro-$%
C^{\ast }$-dynamical system such that $\alpha $ is an inverse limit action.%
\textbf{\ }Then the pro-$C^{\ast }$-algebras $G\times _{\alpha }A\left[ \tau
_{\Gamma }\right] $ and $\lim\limits_{\leftarrow \lambda }G\times _{\alpha
^{\lambda }}A_{\lambda }$ are isomorphic.
\end{corollary}

\begin{remark}
If $\left( G,\alpha ,A\left[ \tau _{\Gamma }\right] \right) $ is a pro-$%
C^{\ast }$-dynamical system such that $\alpha $ is an inverse limit action,%
\textbf{\ }then the notion of full pro-$C^{\ast }$-crossed product in the
sense of Definition 5.4\textbf{\ }coincides to the notion of full crossed
product introduced by \cite{P2, J4}.
\end{remark}

Let $A\left[ \tau _{\Gamma }\right] $ and $B\left[ \tau _{\Gamma ^{\prime }}%
\right] $ be two pro-$C^{\ast }$-algebras. For each $p_{\lambda }\in \Gamma $
and $q_{\delta }\in \Gamma ^{\prime }$, the map $t_{p_{\lambda },q_{\delta
}}:A\left[ \tau _{\Gamma }\right] \otimes _{\text{alg}}B\left[ \tau _{\Gamma
^{\prime }}\right] \rightarrow \lbrack 0,\infty )$ given by 
\begin{equation*}
t_{p_{\lambda },q_{\delta }}\left( z\right) =\sup \{\left\Vert \varphi \circ
\pi _{p_{\lambda },q_{\delta }}\left( z\right) \right\Vert ;\varphi \text{
is a}\ast \text{-representation of }A_{\lambda }\otimes _{\text{alg}}B_{j}\},
\end{equation*}%
where $\pi _{p_{\lambda },q_{\delta }}\left( a\otimes b\right) =\pi
_{\lambda }^{A}\left( a\right) \otimes \pi _{\delta }^{B}\left( b\right) $,
defines a $C^{\ast }$-seminorm on the algebraic tensor product $A\left[ \tau
_{\Gamma }\right] \otimes _{\text{alg}}B\left[ \tau _{\Gamma ^{\prime }}%
\right] $. The completion of $A\left[ \tau _{\Gamma }\right] \otimes _{\text{%
alg}}B\left[ \tau _{\Gamma ^{\prime }}\right] $ with respect to the topology
given by the family of $C^{\ast }$-seminorms $\{t_{p_{\lambda },q_{\delta
}};p_{\lambda }\in \Gamma ,q_{\delta }\in \Gamma ^{\prime }\}$ is a pro-$%
C^{\ast }$-algebra, denoted by $A\left[ \tau _{\Gamma }\right] \otimes
_{\max }B\left[ \tau _{\Gamma ^{\prime }}\right] $, and called the maximal
or projective tensor product of the pro-$C^{\ast }$-algebras $A\left[ \tau
_{\Gamma }\right] $ and $B\left[ \tau _{\Gamma ^{\prime }}\right] $ (see 
\cite[Chapter VII]{F}). Moreover, for each $p_{\lambda }\in \Gamma $ and $%
q_{\delta }\in \Gamma ^{\prime }$, the $C^{\ast }$-algebras $\left( A\left[
\tau _{\Gamma }\right] \otimes _{\max }B\left[ \tau _{\Gamma ^{\prime }}%
\right] \right) _{\left( \lambda ,\delta \right) }$ and $A_{\lambda }\otimes
_{\max }B_{\delta }$ are isomorphic.

\begin{remark}
The trivial action of $G$ on $A\left[ \tau _{\Gamma }\right] $ is an inverse
limit action, and so the full pro-$C^{\ast }$-crossed product of $A\left[
\tau _{\Gamma }\right] $ by the trivial action is isomorphic to $A\left[
\tau _{\Gamma }\right] \otimes _{\max }C^{\ast }(G)$ \cite[Corollary 1.3.9]%
{J2}.
\end{remark}

Let $\left( G,\alpha ,A\left[ \tau _{\Gamma }\right] \right) $ be a pro-$%
C^{\ast }$-dynamical system such that $\alpha $ is \textbf{s}trongly bounded%
\textbf{\ }and\textbf{\ }let $B\left[ \tau _{\Gamma ^{\prime }}\right] $ be
a pro-$C^{\ast }$-algebra. Then $t\mapsto \left( \alpha \otimes \text{id}%
\right) _{t}$, where $\left( \alpha \otimes \text{id}\right) _{t}\left(
a\otimes b\right) =\alpha _{t}\left( a\right) \otimes b$, is a strong
bounded action of $G$ on $A\left[ \tau _{\Gamma }\right] \otimes _{\max }B%
\left[ \tau _{\Gamma ^{\prime }}\right] $.

The following theorem gives an "associativity" between $\times _{\alpha }$
and $\otimes _{\max }$.

\begin{theorem}
Let $\left( G,\alpha ,A\left[ \tau _{\Gamma }\right] \right) $ be a pro-$%
C^{\ast }$-dynamical system such that $\alpha $ is strongly bounded and\ let%
\textbf{\ }$B\left[ \tau _{\Gamma ^{\prime }}\right] $ be a pro-$C^{\ast }$%
-algebra. Then the pro-$C^{\ast }$-algebras $G\times _{\alpha \otimes \text{%
id}}(A\left[ \tau _{\Gamma }\right] \otimes _{\max }B\left[ \tau _{\Gamma
^{\prime }}\right] )$ and $\left( G\times _{\alpha }A\left[ \tau _{\Gamma }%
\right] \right) \otimes _{\max }B\left[ \tau _{\Gamma ^{\prime }}\right] $
are isomorphic\textbf{.}
\end{theorem}

\begin{proof}
Let $\rho _{G\times _{\alpha }A\left[ \tau _{\Gamma }\right] }:G\times
_{\alpha }A\left[ \tau _{\Gamma }\right] \rightarrow M(\left( G\times
_{\alpha }A\left[ \tau _{\Gamma }\right] \right) \otimes _{\max }B\left[
\tau _{\Gamma ^{\prime }}\right] )$ and $\rho _{B}:B\left[ \tau _{\Gamma
^{\prime }}\right] \rightarrow M(\left( G\times _{\alpha }A\left[ \tau
_{\Gamma }\right] \right) \otimes _{\max }B\left[ \tau _{\Gamma ^{\prime }}%
\right] )$ be the canonical maps. Then $\overline{\rho _{G\times _{\alpha }A%
\left[ \tau _{\Gamma }\right] }}\circ \iota _{A}:A\left[ \tau _{\Gamma }%
\right] \rightarrow M(\left( G\times _{\alpha }A\left[ \tau _{\Gamma }\right]
\right) \otimes _{\max }B\left[ \tau _{\Gamma ^{\prime }}\right] )$ and $%
\rho _{B}:B\left[ \tau _{\Gamma ^{\prime }}\right] \rightarrow M(\left(
G\times _{\alpha }A\left[ \tau _{\Gamma }\right] \right) \otimes _{\max }B%
\left[ \tau _{\Gamma ^{\prime }}\right] )$ are nondegenerate pro-$C^{\ast }$%
-morphisms with commuting ranges.

Let $j_{G\times _{\alpha \otimes \text{id}}A\left[ \tau _{\Gamma }\right]
\otimes _{\max }B\left[ \tau _{\Gamma ^{\prime }}\right] }=\overline{\rho
_{G\times _{\alpha }A\left[ \tau _{\Gamma }\right] }}\circ \iota _{A}\otimes
\rho _{B}$ and $j_{G}$ $=\overline{\rho _{G\times _{\alpha }A\left[ \tau
_{\Gamma }\right] }}\circ \iota _{G}$. A simple calculus shows that $\left(
j_{G\times _{\alpha \otimes \text{id}}A\left[ \tau _{\Gamma }\right] \otimes
_{\max }B\left[ \tau _{\Gamma ^{\prime }}\right] },j_{G}\right) $ is a
nondegenerate covariant pro-$C^{\ast }$-morphism from $\left( G,\alpha
\otimes \text{id},A\left[ \tau _{\Gamma }\right] \otimes _{\max }B\left[
\tau _{\Gamma ^{\prime }}\right] \right) $ to $M(\left( G\times _{\alpha }A%
\left[ \tau _{\Gamma }\right] \right) \otimes _{\max }B\left[ \tau _{\Gamma
^{\prime }}\right] )$. Moreover, from%
\begin{eqnarray*}
&&j_{G\times _{\alpha \otimes \text{id}}A\left[ \tau _{\Gamma }\right]
\otimes _{\max }B\left[ \tau _{\Gamma ^{\prime }}\right] }\left( a\otimes
b\right) j_{G}\left( f\right) \\
&=&\overline{\rho _{G\times _{\alpha }A\left[ \tau _{\Gamma }\right] }}%
\left( \iota _{A}\left( a\right) \right) \rho _{B}\left( b\right) \overline{%
\rho _{G\times _{\alpha }A\left[ \tau _{\Gamma }\right] }}\left( \iota
_{G}\left( f\right) \right) \\
&=&\overline{\rho _{G\times _{\alpha }A\left[ \tau _{\Gamma }\right] }}%
\left( \iota _{A}\left( a\right) \right) \overline{\rho _{G\times _{\alpha }A%
\left[ \tau _{\Gamma }\right] }}\left( \iota _{G}\left( f\right) \right)
\rho _{B}\left( b\right) \\
&=&\rho _{G\times _{\alpha }A\left[ \tau _{\Gamma }\right] }\left( \iota
_{A}\left( a\right) \iota _{G}\left( f\right) \right) \rho _{B}\left(
b\right)
\end{eqnarray*}%
for all $a\in A$, for all $b\in B$ and for all $f\in C_{c}(G)$, and taking
into account that $\{\iota _{A}\left( a\right) \iota _{G}\left( f\right)
;a\in A,f\in C_{c}(G)\}$ generates $G\times _{\alpha }A\left[ \tau _{\Gamma }%
\right] $ and $\{\rho _{G\times _{\alpha }A\left[ \tau _{\Gamma }\right]
}\left( z\right) \rho _{B}\left( b\right) ;z\in G\times _{\alpha }A\left[
\tau _{\Gamma }\right] ,b\in B\}$ generates $\left( G\times _{\alpha }A\left[
\tau _{\Gamma }\right] \right) \otimes _{\max }B\left[ \tau _{\Gamma
^{\prime }}\right] $, we conclude that%
\begin{eqnarray*}
&&\overline{\text{span}\{j_{G\times _{\alpha \otimes \text{id}}A\left[ \tau
_{\Gamma }\right] \otimes _{\max }B\left[ \tau _{\Gamma ^{\prime }}\right]
}\left( a\otimes b\right) j_{G}\left( f\right) ;a\in A,b\in b,f\in C_{c}(G)\}%
} \\
&=&\left( G\times _{\alpha }A\left[ \tau _{\Gamma }\right] \right) \otimes
_{\max }B\left[ \tau _{\Gamma ^{\prime }}\right] .
\end{eqnarray*}%
Let $\left( \varphi ,u,\mathcal{H}\right) $ be a nondegenerate covariant
representation of $(G,\alpha \otimes $id$,A\left[ \tau _{\Gamma }\right]
\otimes _{\max }B\left[ \tau _{\Gamma ^{\prime }}\right] )$. Then $\left(
\varphi ,\mathcal{H}\right) $ is a nondegenerate representation of $A\left[
\tau _{\Gamma }\right] \otimes _{\max }B\left[ \tau _{\Gamma ^{\prime }}%
\right] $, and so there is a nondegenerate representation $\left( \varphi
_{\left( \lambda ,\delta \right) },\mathcal{H}\right) $ of $A_{\lambda
}\otimes _{\max }B_{\delta }$ such that $\varphi _{\left( \lambda ,\delta
\right) }\circ \pi _{\left( \lambda ,\delta \right) }^{A\left[ \tau _{\Gamma
}\right] \otimes _{\max }B\left[ \tau _{\Gamma ^{\prime }}\right] }=\varphi $%
. Let $\left( \varphi _{\lambda },\mathcal{H}\right) $ and $\left( \varphi
_{\delta },\mathcal{H}\right) $ be the nondegenerate representations of $%
A_{\lambda }$, respectively $B_{\delta }$ with commuting ranges such that $%
\varphi _{\left( \lambda ,\delta \right) }\left( a\otimes b\right) =\varphi
_{\lambda }\left( a\right) \varphi _{\delta }\left( b\right) $ for all $a\in
A_{\lambda }$ and $b\in B_{\delta }$. Then $\left( \varphi _{\lambda }\circ
\pi _{\lambda }^{A},u,\mathcal{H}\right) $ is a nondegenerate covariant
representation of $\left( G,\alpha ,A\left[ \tau _{\Gamma }\right] \right) $%
, and so there is a nondegenerate representation $\left( \Phi _{1},\mathcal{H%
}\right) $ of $G\times _{\alpha }A\left[ \tau _{\Gamma }\right] $ such that $%
\overline{\Phi _{1}}\circ \iota _{A}=\varphi _{\lambda }\circ \pi _{\lambda
}^{A}$ and $\overline{\Phi _{1}}\circ \iota _{G}=u$. It is easy to check
that $\left( \Phi _{1},\mathcal{H}\right) $ and $\left( \Phi _{2},\mathcal{H}%
\right) $, where $\Phi _{2}=\varphi _{\delta }\circ \pi _{\delta }^{B}$, are
nondegenerate representations of $G\times _{\alpha }A\left[ \tau _{\Gamma }%
\right] $ respectively $B\left[ \tau _{\Gamma ^{\prime }}\right] $ with
commuting ranges. Let $\left( \Phi ,\mathcal{H}\right) $ be the
nondegenerate representation of $\left( G\times _{\alpha }A\left[ \tau
_{\Gamma }\right] \right) \otimes _{\max }B\left[ \tau _{\Gamma ^{\prime }}%
\right] $ given by $\Phi \left( z\otimes b\right) =\Phi _{1}\left( z\right)
\Phi _{2}\left( b\right) $. Then 
\begin{eqnarray*}
&&\overline{\Phi }\left( j_{G\times _{\alpha \otimes \text{id}}A\left[ \tau
_{\Gamma }\right] \otimes _{\max }B\left[ \tau _{\Gamma ^{\prime }}\right]
}\left( a\otimes b\right) \right) \\
&=&\overline{\Phi }\left( \overline{\rho _{G\times _{\alpha }A\left[ \tau
_{\Gamma }\right] }}\left( \iota _{A}\left( a\right) \right) \rho _{B}\left(
b\right) \right) =\overline{\Phi _{1}}\left( \left( \iota _{A}\left(
a\right) \right) \overline{\Phi _{2}}(\rho _{B}\left( b\right) \right) \\
&=&\left( \varphi _{\lambda }\circ \pi _{\lambda }^{A}\left( a\right)
\right) \left( \varphi _{\delta }\circ \pi _{\delta }^{B}\left( b\right)
\right) =\varphi _{\left( \lambda ,\delta \right) }\left( \pi _{\lambda
}^{A}\left( a\right) \otimes \pi _{\delta }^{B}\left( b\right) \right) \\
&=&\varphi _{\left( \lambda ,\delta \right) }\circ \pi _{\left( \lambda
,\delta \right) }^{A\left[ \tau _{\Gamma }\right] \otimes _{\max }B\left[
\tau _{\Gamma ^{\prime }}\right] }\left( a\otimes b\right) =\varphi \left(
a\otimes b\right)
\end{eqnarray*}%
for all $a\in A$ and $b\in B$, and%
\begin{equation*}
\overline{\Phi }\left( j_{G}\left( f\right) \right) =\overline{\Phi }\left( 
\overline{\rho _{G\times _{\alpha }A\left[ \tau _{\Gamma }\right] }}\circ
\iota _{G}\left( f\right) \right) =\overline{\Phi _{1}}\left( \iota
_{G}\left( f\right) \right) =u\left( f\right)
\end{equation*}%
for all $f\in C_{c}\left( G\right) $. Therefore, by Definition 5.4 and
Corollary 5.7, the pro-$C^{\ast }$-algebras $G\times _{\alpha \otimes \text{%
id}}A\left[ \tau _{\Gamma }\right] \otimes _{\max }B\left[ \tau _{\Gamma
^{\prime }}\right] $ and $\left( G\times _{\alpha }A\left[ \tau _{\Gamma }%
\right] \right) \otimes _{\max }B\left[ \tau _{\Gamma ^{\prime }}\right] $
are isomorphic\textbf{.}
\end{proof}

\begin{definition}
We say that $\left( G,\alpha ,A\left[ \tau _{\Gamma }\right] \right) $ and $%
\left( G,\beta ,B\left[ \tau _{\Gamma ^{\prime }}\right] \right) $ are
conjugate if there is a pro-$C^{\ast }$-isomorphism $\varphi :A\left[ \tau
_{\Gamma }\right] \rightarrow B\left[ \tau _{\Gamma ^{\prime }}\right] $
such that $\varphi \circ \alpha _{t}=\beta _{t}\circ \varphi $ for all $t\in
G.$
\end{definition}

\begin{remark}
If $\left( G,\alpha ,A\left[ \tau _{\Gamma }\right] \right) $ and $\left(
G,\beta ,B\left[ \tau _{\Gamma ^{\prime }}\right] \right) $ are conjugate
and $\alpha $ is strongly bounded, then $\beta $ is strongly bounded too.
\end{remark}

\begin{proposition}
Let $\left( G,\alpha ,A\left[ \tau _{\Gamma }\right] \right) $ and $\left(
G,\beta ,B\left[ \tau _{\Gamma ^{\prime }}\right] \right) $ be two pro-$%
C^{\ast }$-dynamical systems such that $\alpha $ and $\beta $ are strongly
bounded.\textbf{\ }If\textbf{\ }$\left( G,\alpha ,A\left[ \tau _{\Gamma }%
\right] \right) $ and $\left( G,\beta ,B\left[ \tau _{\Gamma ^{\prime }}%
\right] \right) $ are conjugate, then the full pro-$C^{\ast }$-crossed
products associated to these pro-$C^{\ast }$-dynamical systems are
isomorphic.
\end{proposition}

\begin{proof}
Let $\varphi :A\left[ \tau _{\Gamma }\right] \rightarrow B\left[ \tau
_{\Gamma ^{\prime }}\right] $ be a pro-$C^{\ast }$-isomorphism such that $%
\varphi \circ \alpha _{t}=\beta _{t}\circ \varphi $ for all $t\in G$. It is
easy to check that $\left( \iota _{B}\circ \varphi ,\iota _{G,B}\right) $ is
a nondegenerate covariant morphism from $\left( G,\alpha ,A\left[ \tau
_{\Gamma }\right] \right) $ to $G\times _{\beta }B\left[ \tau _{\Gamma
^{\prime }}\right] $, where $\left( \iota _{B},\iota _{G,B}\right) $ is the
covariant morphism from $\left( G,\beta ,B\left[ \tau _{\Gamma ^{\prime }}%
\right] \right) $ to $G\times _{\beta }B\left[ \tau _{\Gamma ^{\prime }}%
\right] $ which defines the full pro-$C^{\ast }$-crossed product of $B\left[
\tau _{\Gamma ^{\prime }}\right] $ by $\beta $. Then, by Proposition 5.6,
there is a nondegenerate pro-$C^{\ast }$-morphism $\Phi :G\times _{\alpha }A%
\left[ \tau _{\Gamma }\right] \rightarrow M\left( G\times _{\beta }B\left[
\tau _{\Gamma ^{\prime }}\right] \right) $ such that $\overline{\Phi }\circ
\iota _{A}=\iota _{B}\circ \varphi $ and $\overline{\Phi }\circ \iota
_{G,A}=\iota _{G,B}$. Moreover, using Definition 5.4, it is easy to check
that $\Phi \left( G\times _{\alpha }A\left[ \tau _{\Gamma }\right] \right)
\subseteq G\times _{\beta }B\left[ \tau _{\Gamma ^{\prime }}\right] $. In
the same manner, we obtain a nondegenerate pro-$C^{\ast }$-morphism $\Psi
:G\times _{\beta }B\left[ \tau _{\Gamma ^{\prime }}\right] \rightarrow
M\left( G\times _{\alpha }A\left[ \tau _{\Gamma }\right] \right) $ such that 
$\overline{\Psi }\circ \iota _{B}=\iota _{A}\circ \varphi ^{-1}$ and $%
\overline{\Psi }\circ \iota _{G,B}=\iota _{G,A}$.

From%
\begin{equation*}
\left( \Phi \circ \Psi \right) \left( \iota _{B}\left( b\right) \iota
_{G,B}\left( f\right) \right) =\Phi \left( \iota _{A}\circ \varphi
^{-1}\left( b\right) \iota _{G,A}\left( f\right) \right) =\iota _{B}\left(
b\right) \iota _{G,B}\left( f\right)
\end{equation*}%
and 
\begin{equation*}
\left( \Psi \circ \Phi \right) \left( \iota _{A}\left( a\right) \iota
_{G,A}\left( f\right) \right) =\Psi \left( \iota _{B}\circ \varphi \left(
a\right) \iota _{G,B}\left( f\right) \right) =\iota _{A}\left( a\right)
\iota _{G,A}\left( f\right)
\end{equation*}%
for all $b\in B\left[ \tau _{\Gamma ^{\prime }}\right] ,$ $a\in A\left[ \tau
_{\Gamma }\right] $ and $f\in C_{c}(G)$ and Definition 5.4, we deduce that $%
\Phi $ and $\Psi $ are pro-$C^{\ast }$-isomorphisms.
\end{proof}

\begin{corollary}
Let $\left( G,\alpha ,A\left[ \tau _{\Gamma }\right] \right) $ be a pro-$%
C^{\ast }$-dynamical system such that $\alpha $ is strongly bounded.

\begin{enumerate}
\item Pro-$C^{\ast }$-algebras $G\times _{\alpha }A\left[ \tau _{\Gamma }%
\right] $ and $G\times _{\alpha }A\left[ \tau _{\Gamma ^{G}}\right] $ are
isomorphic.

\item $A\left[ \tau _{\Gamma }\right] $ is isomorphic to a pro-$C^{\ast }$%
-subalgebra of $M(G\times _{\alpha }A\left[ \tau _{\Gamma }\right] )$.
\end{enumerate}
\end{corollary}

\section{The reduced pro-$C^{\ast }$-crossed product}

Let $A\left[ \tau _{\Gamma }\right] $ and $B\left[ \tau _{\Gamma ^{\prime }}%
\right] $ be two pro-$C^{\ast }$-algebras. For each $p_{\lambda }\in \Gamma $
and $q_{\delta }\in \Gamma ^{\prime }$, the map $\vartheta _{p_{\lambda
},q_{\delta }}:A\left[ \tau _{\Gamma }\right] \otimes _{\text{alg}}B\left[
\tau _{\Gamma ^{\prime }}\right] \rightarrow \lbrack 0,\infty )$ given by%
\begin{equation*}
\vartheta _{p_{\lambda },q_{\delta }}\left( z\right) =\sup \{\left\Vert
\left( \varphi \otimes \psi \right) \left( z\right) \right\Vert ;\varphi \in 
\mathcal{R}_{\lambda }\left( A\left[ \tau _{\Gamma }\right] \right) \text{, }%
\psi \in \mathcal{R}_{\delta }\left( B\left[ \tau _{\Gamma ^{\prime }}\right]
\right) \}
\end{equation*}%
defines a $C^{\ast }$-seminorm on the algebraic tensor product $A\left[ \tau
_{\Gamma }\right] \otimes _{\text{alg}}B\left[ \tau _{\Gamma ^{\prime }}%
\right] $. The completion of $A\left[ \tau _{\Gamma }\right] \otimes _{\text{%
alg}}B\left[ \tau _{\Gamma ^{\prime }}\right] $ with respect to the topology
given by the family of $C^{\ast }$-seminorms $\{\vartheta _{p_{\lambda
},q_{\delta }};p_{\lambda }\in \Gamma ,q_{\delta }\in \Gamma ^{\prime }\}$
is a pro-$C^{\ast }$-algebra, denoted by $A\left[ \tau _{\Gamma }\right]
\otimes B\left[ \tau _{\Gamma ^{\prime }}\right] $, and called the minimal
or injective tensor product of the pro-$C^{\ast }$-algebras $A\left[ \tau
_{\Gamma }\right] $ and $B\left[ \tau _{\Gamma ^{\prime }}\right] $ (see 
\cite[Chapter VII]{F}). Moreover, for each $p_{\lambda }\in \Gamma $ and $%
q_{\delta }\in \Gamma ^{\prime }$, the $C^{\ast }$-algebras $\left( A\left[
\tau _{\Gamma }\right] \otimes _{\min }B\left[ \tau _{\Gamma ^{\prime }}%
\right] \right) _{\left( \lambda ,\delta \right) }$ and $A_{\lambda }\otimes
_{\min }B_{\delta }$ are isomorphic.

Let $\left( G,\alpha ,A\left[ \tau _{\Gamma }\right] \right) $ be a pro-$%
C^{\ast }$-dynamical system such that $\alpha $ is strongly bounded.\textbf{%
\ }Since $\alpha $ is strongly bounded, for each $a\in A$, the map $t\mapsto
\alpha _{t^{-1}}\left( a\right) $ defines an element in $C_{b}(G,A\left[
\tau _{\Gamma }\right] )$, the pro-$C^{\ast }$-algebra of all bounded
continuous functions from $G$ to $A\left[ \tau _{\Gamma }\right] $, and so
there is a map $\widetilde{\alpha }:A\left[ \tau _{\Gamma }\right]
\rightarrow C_{b}(G,A\left[ \tau _{\Gamma }\right] )$ given by $\widetilde{%
\alpha }\left( a\right) \left( t\right) =\alpha _{t^{-1}}\left( a\right) $.

\begin{lemma}
Let $\left( G,\alpha ,A\left[ \tau _{\Gamma }\right] \right) $ be a pro-$%
C^{\ast }$-dynamical system such that $\alpha $ is strongly bounded. Then $%
\widetilde{\alpha }$ is a nondegenerate faithful pro-$C^{\ast }$-morphism
from $A\left[ \tau _{\Gamma }\right] $ to $M(A\left[ \tau _{\Gamma }\right]
\otimes _{\min }C_{0}\left( G\right) )$ with closed range. Moreover, if $%
\alpha $ is an inverse limit action, then $\widetilde{\alpha }$ is an
inverse limit pro-$C^{\ast }$-morphism.
\end{lemma}

\begin{proof}
Clearly, $\widetilde{\alpha }$ is a $\ast $-morphism. For each $p_{\lambda
}\in \Gamma $, there is $p_{\mu }\in \Gamma $ such that%
\begin{equation*}
p_{\lambda }\left( a\right) =p_{\lambda }\left( \alpha _{e}\left( a\right)
\right) \leq \sup \{p_{\lambda }\left( \alpha _{t}\left( a\right) \right)
;t\in G\}=p_{\lambda ,C_{b}(G,A\left[ \tau _{\Gamma }\right] )}\left( 
\widetilde{\alpha }\left( a\right) \right) \leq p_{\mu }\left( a\right)
\end{equation*}%
for all $a\in A$. Therefore, $\widetilde{\alpha }$ is an injective pro-$%
C^{\ast }$-morphism with closed range. By \cite[p. 76]{J2}, $C_{b}(G,A\left[
\tau _{\Gamma }\right] )$ can be identified to a pro-$C^{\ast }$-subalgebra
of $M(A\left[ \tau _{\Gamma }\right] \otimes C_{0}\left( G\right) )$, and
then $\widetilde{\alpha }$ can be regarded as a pro-$C^{\ast }$-morphism
from $A\left[ \tau _{\Gamma }\right] $ to $M(A\left[ \tau _{\Gamma }\right]
\otimes C_{0}\left( G\right) )$.

To show that $\widetilde{\alpha }$ is nondegenerate, let $\{e_{i}\}_{i\in I}$
be an approximate unit for $A\left[ \tau _{\Gamma }\right] $. In the same
manner as in \cite[Proposition 5.1.5]{V}, we show that $\{\widetilde{\alpha }%
\left( e_{i}\right) \}_{i\in I}$ is strictly convergent. Indeed, let $a\in
A, $ $f\in C_{c}\left( G\right) $ and $p_{\lambda }\in \Gamma $. Then%
\begin{eqnarray*}
&&p_{\lambda ,C_{b}(G,A\left[ \tau _{\Gamma }\right] )}\left( \widetilde{%
\alpha }\left( e_{i}\right) \left( a\otimes f\right) -a\otimes f\right) \\
&=&\sup \{p_{\lambda }\left( \alpha _{t^{-1}}\left( e_{i}\right) af\left(
t\right) -af\left( t\right) \right) ;t\in G\} \\
&\leq &\left\Vert f\right\Vert _{\infty }\sup \{p_{\lambda }\left( \alpha
_{t^{-1}}\left( e_{i}\alpha _{t}\left( a\right) -\alpha _{t}\left( a\right)
\right) \right) ;t\in \text{supp}\left( f\right) \} \\
&\leq &\left\Vert f\right\Vert _{\infty }\sup \{p_{\mu }\left( e_{i}\alpha
_{t}\left( a\right) -\alpha _{t}\left( a\right) \right) ;t\in \text{supp}%
\left( f\right) \}
\end{eqnarray*}%
for some $p_{\mu }\in \Gamma $. For each $i\in I$, consider the function $%
f_{i}:G\rightarrow \mathbb{C},$ $f_{i}(t)=$ $p_{\mu }\left( e_{i}\alpha
_{t}\left( a\right) -\alpha _{t}\left( a\right) \right) $. Clearly, $%
\{f_{i}\}_{i\in I}$ is a net of continuous functions on $G$ which is
uniformly bounded and equicontinuous. Then, by Arzel\`{a}--Ascoli's theorem,
it is uniformly convergent on compact subsets of $G$. Therefore, $\{%
\widetilde{\alpha }\left( e_{i}\right) \}_{i\in I}$ is strictly convergent,
and so the pro-$C^{\ast }$-morphism $\widetilde{\alpha }$ is nondegenerate.

Suppose that $\alpha _{t}=\lim\limits_{\leftarrow \lambda }\alpha
_{t}^{\lambda }$ for each $t\in G$. Then $\left( \widetilde{\alpha ^{\lambda
}}\right) _{\lambda }$ is an inverse system of $C^{\ast }$-morphisms and $%
\widetilde{\alpha }=\lim\limits_{\leftarrow \lambda }\widetilde{\alpha
^{\lambda }}$.
\end{proof}

Let $\varphi :A\left[ \tau _{\Gamma }\right] \rightarrow M(B\left[ \tau
_{\Gamma ^{\prime }}\right] )$ be a nondegenerate pro-$C^{\ast }$-morphism
and let $M:$ $C_{0}(G)\rightarrow L(L^{2}(G))$ be the representation by
multiplication operators. Then there is a nondegenerate pro-$C^{\ast }$%
-morphism $\varphi \otimes M:A\left[ \tau _{\Gamma }\right] \otimes _{\min
}C_{0}(G)\rightarrow M(B\left[ \tau _{\Gamma ^{\prime }}\right] \otimes
_{\min }\mathcal{K}(L^{2}(G))$ such that $\left( \varphi \otimes M\right)
\left( a\otimes f\right) =\varphi \left( a\right) \otimes M_{f}$, and since $%
\widetilde{\alpha }$ is a nondegenerate pro-$C^{\ast }$-morphism from $A%
\left[ \tau _{\Gamma }\right] $ to $M(A\left[ \tau _{\Gamma }\right] \otimes
_{\min }C_{0}\left( G\right) )$, $\widetilde{\varphi }=$ $\overline{\varphi
\otimes M}\circ \widetilde{\alpha }$ is a nondegenerate pro-$C^{\ast }$%
-morphism from $A\left[ \tau _{\Gamma }\right] $ to $M(B\left[ \tau _{\Gamma
^{\prime }}\right] \otimes _{\min }\mathcal{K}(L^{2}(G))$.

Let $\lambda _{G}:G\rightarrow \mathcal{U}(L^{2}(G))$ be the left
representation of $G$ on $L^{2}(G)$ given by $\left( \lambda _{G}\right)
_{t}\left( \xi \right) \left( s\right) =\xi \left( t^{-1}s\right) $. Then $%
1\otimes \lambda _{G}:G\rightarrow \mathcal{U}\left( M(B\left[ \tau _{\Gamma
^{\prime }}\right] \otimes \mathcal{K}(L^{2}(G))\right) $, where $\left(
1\otimes \lambda _{G}\right) _{t}\left( b\otimes \xi \right) \left( s\right)
=b\xi \left( t^{-1}s\right) $, is a strict continuous group morphism from $G$
to $\mathcal{U}\left( M(B\left[ \tau _{\Gamma ^{\prime }}\right] \otimes
_{\min }\mathcal{K}(L^{2}(G))\right) $, and $\left( \widetilde{\varphi }%
,1\otimes \lambda _{G}\right) $ is a nondegenerate covariant morphism of $%
\left( G,\alpha ,A\left[ \tau _{\Gamma }\right] \right) $ to $B\left[ \tau
_{\Gamma ^{\prime }}\right] \otimes _{\min }\mathcal{K}(L^{2}(G)$. By
Proposition 5.6, there is a unique nondegenerate pro-$C^{\ast }$-morphism $%
\widetilde{\varphi }\times \left( 1\otimes \lambda _{G}\right) :G\times
_{\alpha }A\left[ \tau _{\Gamma }\right] \rightarrow M(B\left[ \tau _{\Gamma
^{\prime }}\right] \otimes _{\min }\mathcal{K}(L^{2}(G))$ such that $%
\widetilde{\varphi }\times \left( 1\otimes \lambda _{G}\right) \circ \iota
_{A}=\widetilde{\varphi }$ and $\widetilde{\varphi }\times \left( 1\otimes
\lambda _{G}\right) \circ \iota _{G}=1\otimes \lambda _{G}$.

If $\varphi =$id$_{A}$, the nondegenerate pro-$C^{\ast }$-morphism $%
\widetilde{\text{id}}_{A}\times \left( 1\otimes \lambda _{G}\right) :$ $%
G\times _{\alpha }A\left[ \tau _{\Gamma }\right] \rightarrow M(A\left[ \tau
_{\Gamma }\right] \otimes _{\min }\mathcal{K}(L^{2}(G))$ is denoted by $%
\Lambda _{A}^{G}$. It is easy to check that $\widetilde{\varphi }\times
\left( 1\otimes \lambda _{G}\right) =\overline{\varphi \otimes \text{id}_{%
\mathcal{K}(L^{2}(G))}}\circ \Lambda _{A}^{G}$.

If $\alpha $ is an inverse limit action, $\alpha
_{t}=\lim\limits_{\leftarrow \lambda }\alpha _{t}^{\lambda }$ for each $t\in
G$, then it is easy to check that $\Lambda _{A}^{G}$ is an inverse limit pro-%
$C^{\ast }$-morphism, $\Lambda _{A}^{G}=\lim\limits_{\leftarrow \lambda
}\Lambda _{A_{\lambda }}^{G}.$

\begin{definition}
The reduced pro-$C^{\ast }$-crossed product of $A\left[ \tau _{\Gamma }%
\right] $ by $\alpha $ is the pro-$C^{\ast }$-subalgebra $G\times _{\alpha
,r}A\left[ \tau _{\Gamma }\right] $ of $M(A\left[ \tau _{\Gamma }\right]
\otimes _{\min }\mathcal{K}(L^{2}(G)))$ generated by the range of $\Lambda
_{A}^{G}$.
\end{definition}

\begin{remark}
From%
\begin{equation*}
\Lambda _{A}^{G}\left( \iota _{A}\left( a\right) \iota _{G}\left( f\right)
\right) =\left( \overline{\text{id}_{A}\otimes M}\circ \widetilde{\alpha }%
\right) \left( a\right) \left( 1\otimes \lambda _{G}\right) \left( f\right) =%
\widetilde{\alpha }\left( a\right) \left( 1\otimes \lambda _{G}\left(
f\right) \right)
\end{equation*}%
for all $a\in A$ and for all $f\in C_{c}\left( G\right) $, and taking into
account that $G\times _{\alpha }A\left[ \tau _{\Gamma }\right] $ is
generated by $\{\iota _{A}\left( a\right) \iota _{G}\left( f\right) ;a\in A,$
$f\in C_{c}\left( G\right) \}$, we conclude that $G\times _{\alpha ,r}A\left[
\tau _{\Gamma }\right] $ is the pro-$C^{\ast }$-subalgebra of $M(A\left[
\tau _{\Gamma }\right] \otimes _{\min }\mathcal{K}(L^{2}(G)))$ generated by $%
\{\widetilde{\alpha }\left( a\right) \left( 1\otimes \lambda _{G}\left(
f\right) \right) ;a\in A,$ $f\in C_{c}\left( G\right) \}.$
\end{remark}

\begin{remark}
If $\alpha $ is an inverse limit action, $\alpha
_{t}=\lim\limits_{\leftarrow \lambda }\alpha _{t}^{\lambda }$ for each $t\in
G$, then%
\begin{equation*}
G\times _{\alpha ,r}A\left[ \tau _{\Gamma }\right] =\overline{\Lambda
_{A}^{G}\left( G\times _{\alpha }A\left[ \tau _{\Gamma }\right] \right) }%
=\lim\limits_{\leftarrow \lambda }\overline{\Lambda _{A_{\lambda
}}^{G}\left( G\times _{\alpha ^{\lambda }}A_{\lambda }\right) }%
=\lim\limits_{\leftarrow \lambda }G\times _{\alpha ^{\lambda },r}A_{\lambda }
\end{equation*}%
and moreover, for each $p_{\lambda }\in \Gamma $, the $C^{\ast }$-algebras $%
\left( G\times _{\alpha ,r}A\left[ \tau _{\Gamma }\right] \right) _{\lambda }
$ and $G\times _{\alpha ^{\lambda },r}A_{\lambda }$ are isomorphic.
\end{remark}

\begin{remark}
Since the trivial action of a locally compact group $G$ on a pro-$C^{\ast }$%
-algebra $A\left[ \tau _{\Gamma }\right] $ is an inverse limit action, the
reduced pro-$C^{\ast }$-crossed product of $A\left[ \tau _{\Gamma }\right] $
by the trivial action is the inverse limit of the reduced crossed products
of $A_{\lambda }$ by the trivial action, and so it is isomorphic to the pro-$%
C^{\ast }$-algebra $A\left[ \tau _{\Gamma }\right] \otimes _{\min
}C_{r}^{\ast }\left( G\right) $.
\end{remark}

Let $\left( G,\alpha ,A\left[ \tau _{\Gamma }\right] \right) $ be a pro-$%
C^{\ast }$-dynamical system such that $\alpha $ is strongly bounded and\ let%
\textbf{\ }$B\left[ \tau _{\Gamma ^{\prime }}\right] $ be a pro-$C^{\ast }$%
-algebra. Then $t\mapsto \left( \alpha \otimes \text{id}\right) _{t}$, where 
$\left( \alpha \otimes \text{id}\right) _{t}\left( a\otimes b\right) =\alpha
_{t}\left( a\right) \otimes b$, is a strong bounded\textbf{\ }action of $G$
on $A\left[ \tau _{\Gamma }\right] \otimes B\left[ \tau _{\Gamma ^{\prime }}%
\right] $.

The following theorem gives an "associativity" between $\times _{\alpha ,r}$
and $\otimes _{\min }$.

\begin{theorem}
Let $\left( G,\alpha ,A\left[ \tau _{\Gamma }\right] \right) $ be a pro-$%
C^{\ast }$-dynamical system such that $\alpha $ is strongly bounded and\ let 
$B\left[ \tau _{\Gamma ^{\prime }}\right] $ be a pro-$C^{\ast }$-algebra.
Then the pro-$C^{\ast }$-algebras $G\times _{\alpha \otimes \text{id},r}(A%
\left[ \tau _{\Gamma }\right] \otimes _{\min }B\left[ \tau _{\Gamma ^{\prime
}}\right] )$ and $\left( G\times _{\alpha ,r}A\left[ \tau _{\Gamma }\right]
\right) \otimes _{\min }B\left[ \tau _{\Gamma ^{\prime }}\right] $ are
isomorphic\textbf{.}
\end{theorem}

\begin{proof}
The map id$_{A}\otimes \sigma _{B,\mathcal{K}(L^{2}(G))}:$ $A\left[ \tau
_{\Gamma }\right] \otimes _{\min }B\left[ \tau _{\Gamma ^{\prime }}\right]
\otimes _{\min }\mathcal{K}(L^{2}(G))\rightarrow A\left[ \tau _{\Gamma }%
\right] \otimes _{\min }\mathcal{K}(L^{2}(G))\otimes _{\min }B\left[ \tau
_{\Gamma ^{\prime }}\right] $ given by%
\begin{equation*}
\text{id}_{A}\otimes \sigma _{B,\mathcal{K}(L^{2}(G))}\left( a\otimes
b\otimes T\right) =a\otimes T\otimes b
\end{equation*}%
is a pro-$C^{\ast }$-isomorphism. Moreover, id$_{A}\otimes \sigma _{B,%
\mathcal{K}(L^{2}(G))}$ is an inverse limit of $C^{\ast }$-isomorphisms. From%
\begin{equation*}
\overline{\text{id}_{A}\otimes \sigma _{B,\mathcal{K}(L^{2}(G))}}\left( 
\widetilde{\alpha \otimes \text{id}}\left( a\otimes b\right) \left( 1\otimes
\lambda _{G}(f\right) \right) =\widetilde{\alpha }\left( a\right) \left(
1_{M(A\left[ \tau _{\Gamma }\right] )}\otimes \lambda _{G}(f\right) \otimes
b)
\end{equation*}%
for all $a\in A$, for all $b\in B$ and for all $f\in C_{c}(G)$, and Remark
6.3, we deduce that%
\begin{equation*}
\overline{\text{id}_{A}\otimes \sigma _{B,\mathcal{K}(L^{2}(G))}}|_{G\times
_{\alpha \otimes \text{id},r}A\left[ \tau _{\Gamma }\right] \otimes _{\min }B%
\left[ \tau _{\Gamma ^{\prime }}\right] }
\end{equation*}%
is a pro-$C^{\ast }$-isomorphism from $G\times _{\alpha \otimes \text{id}%
,r}(A\left[ \tau _{\Gamma }\right] \otimes _{\min }B\left[ \tau _{\Gamma
^{\prime }}\right] )$ onto \ $\left( G\times _{\alpha ,r}A\left[ \tau
_{\Gamma }\right] \right) \otimes _{\min }B\left[ \tau _{\Gamma ^{\prime }}%
\right] $. Therefore, the pro-$C^{\ast }$-algebras $\left( G\times _{\alpha
,r}A\left[ \tau _{\Gamma }\right] \right) \otimes _{\min }B\left[ \tau
_{\Gamma ^{\prime }}\right] $ and $G\times _{\alpha \otimes \text{id},r}(A%
\left[ \tau _{\Gamma }\right] \otimes _{\min }B\left[ \tau _{\Gamma ^{\prime
}}\right] )$ are isomorphic\textbf{.}
\end{proof}

\begin{proposition}
Let $\left( G,\alpha ,A\left[ \tau _{\Gamma }\right] \right) $ and $\left(
G,\beta ,B\left[ \tau _{\Gamma ^{\prime }}\right] \right) $ be two pro-$%
C^{\ast }$-dynamical systems such that $\alpha $ and $\beta $ are strongly
bounded.\textbf{\ }If $\left( G,\alpha ,A\left[ \tau _{\Gamma }\right]
\right) $ and $\left( G,\beta ,B\left[ \tau _{\Gamma ^{\prime }}\right]
\right) $ are conjugate, then the reduced pro-$C^{\ast }$-crossed products
associated to these pro-$C^{\ast }$-dynamical systems are isomorphic.
\end{proposition}

\begin{proof}
Let $\varphi :A\left[ \tau _{\Gamma }\right] \rightarrow B\left[ \tau
_{\Gamma ^{\prime }}\right] $ be a pro-$C^{\ast }$-isomorphism such that $%
\varphi \circ \alpha _{t}=\beta _{t}\circ \varphi $ for all $t\in G$. It is
easy to check that $\overline{\varphi \otimes \text{id}_{\mathcal{K}%
(L^{2}(G))}}\circ \widetilde{\alpha }=\widetilde{\beta }\circ \varphi $. From%
\begin{equation*}
\overline{\varphi \otimes \text{id}_{\mathcal{K}(L^{2}(G))}}\left( 
\widetilde{\alpha }\left( a\right) \left( 1\otimes \lambda _{G}\left(
f\right) \right) \right) =\widetilde{\beta }\left( \varphi \left( a\right)
\right) \left( 1\otimes \lambda _{G}\left( f\right) \right)
\end{equation*}%
for all $a\in A$ and for all $f\in C_{c}\left( G\right) $, and taking into
account that 
\begin{equation*}
\overline{\text{span}\{\widetilde{\alpha }\left( a\right) \left( 1\otimes
\lambda _{G}\left( f\right) \right) ;a\in A,f\in C_{c}\left( G\right) \}}%
=G\times _{\alpha ,r}A\left[ \tau _{\Gamma }\right]
\end{equation*}%
and 
\begin{equation*}
\overline{\text{span}\{\widetilde{\beta }\left( a\right) \left( 1\otimes
\lambda _{G}\left( f\right) \right) ;a\in B,f\in C_{c}\left( G\right) \}}%
=G\times _{\beta ,r}B\left[ \tau _{\Gamma ^{\prime }}\right] ,
\end{equation*}%
we conclude that $\Phi _{1}=\overline{\varphi \otimes \text{id}_{\mathcal{K}%
(L^{2}(G))}}|_{G\times _{\alpha ,r}A\left[ \tau _{\Gamma }\right] }$ is a
pro-$C^{\ast }$-morphism from $G\times _{\alpha ,r}A\left[ \tau _{\Gamma }%
\right] $ to $G\times _{\beta ,r}B\left[ \tau _{\Gamma ^{\prime }}\right] $.

In the same manner, we conclude that $\Phi _{2}=\overline{\varphi
^{-1}\otimes \text{id}_{\mathcal{K}(L^{2}(G))}}|_{G\times _{\beta ,r}B\left[
\tau _{\Gamma ^{\prime }}\right] }$ is a pro-$C^{\ast }$-morphism from $%
G\times _{\beta ,r}B\left[ \tau _{\Gamma ^{\prime }}\right] $ to $G\times
_{\alpha ,r}A\left[ \tau _{\Gamma }\right] $. Moreover, $\Phi _{1}\circ \Phi
_{2}=$id$_{G\times _{\beta ,r}B\left[ \tau _{\Gamma ^{\prime }}\right] }$
and $\Phi _{2}\circ \Phi _{1}=$id$_{G\times _{\alpha ,r}A\left[ \tau
_{\Gamma }\right] }$, since%
\begin{equation*}
\Phi _{1}\circ \Phi _{2}\left( \widetilde{\beta }\left( b\right) \left(
1\otimes \lambda _{G}\left( f\right) \right) \right) =\widetilde{\beta }%
\left( b\right) \left( 1\otimes \lambda _{G}\left( f\right) \right)
\end{equation*}%
for all $b\in B$ and for all $f\in C_{c}\left( G\right) $ and 
\begin{equation*}
\Phi _{2}\circ \Phi _{1}\left( \widetilde{\alpha }\left( a\right) \left(
1\otimes \lambda _{G}\left( f\right) \right) \right) =\widetilde{\alpha }%
\left( a\right) \left( 1\otimes \lambda _{G}\left( f\right) \right)
\end{equation*}%
for all $a\in A$ and for all $f\in C_{c}\left( G\right) $. Therefore, the
pro-$C^{\ast }$-algebras $G\times _{\alpha ,r}A\left[ \tau _{\Gamma }\right] 
$ and $G\times _{\beta ,r}B\left[ \tau _{\Gamma ^{\prime }}\right] $ are
isomorphic.
\end{proof}

\begin{corollary}
Let $\left( G,\alpha ,A\left[ \tau _{\Gamma }\right] \right) $ be a pro-$%
C^{\ast }$-dynamical systems such that $\alpha $ is strongly bounded. Then
the pro-$C^{\ast }$-algebras $G\times _{\alpha ,r}A\left[ \tau _{\Gamma }%
\right] $ and $G\times _{\alpha ,r}A\left[ \tau _{\Gamma ^{G}}\right] $ are
isomorphic.
\end{corollary}

\begin{remark}
If $\alpha $ is an action of an amenable locally compact group $G$ on a $%
C^{\ast }$-algebra $A$, then the $C^{\ast }$-morphism $\Lambda _{A}^{G}$ is
injective and the full crossed product $A$ by $\alpha $ is isomorphic to the
reduced crossed product of $A$ by $\alpha $. If $\alpha $ is an inverse
limit action of an amenable locally compact group $G$ on a pro-$C^{\ast }$%
-algebra $A\left[ \tau _{\Gamma }\right] $, $\alpha
_{t}=\lim\limits_{\leftarrow \lambda }\alpha _{t}^{\lambda }$ for each $t\in
G$, then $\Lambda _{A}^{G}$ $=\lim\limits_{\leftarrow \lambda }\Lambda
_{A_{\lambda }}^{G}$, and so $\Lambda _{A}^{G}$ is an injective pro-$C^{\ast
}$-morphism with closed range. Therefore, if $G$ is amenable and $\alpha $
is an inverse limit action, then the full pro-$C^{\ast }$-crossed product of 
$A\left[ \tau _{\Gamma }\right] $ by $\alpha $ is isomorphic to the reduced
pro-$C^{\ast }$-crossed product of $A\left[ \tau _{\Gamma }\right] $ by $%
\alpha $.
\end{remark}

\begin{proposition}
Let $\left( G,\alpha ,A\left[ \tau _{\Gamma }\right] \right) $ be a pro-$%
C^{\ast }$-dynamical system such that $\alpha $ is strongly bounded. If $G$
is amenable then the full pro-$C^{\ast }$-crossed product of $A\left[ \tau
_{\Gamma }\right] $ by $\alpha $ is isomorphic to the reduced pro-$C^{\ast }$%
-crossed product of $A\left[ \tau _{\Gamma }\right] $ by $\alpha .$
\end{proposition}

\begin{proof}
It follows from Corollaries 5.20 and 6.8, and Remark 6.9.
\end{proof}

\end{document}